\documentclass[12pt,bezier]{article}
\usepackage{times}
\usepackage{booktabs}
\usepackage{pifont}
\usepackage{floatrow}
\usepackage{caption}
\usepackage{mathrsfs}
\usepackage[fleqn]{amsmath}
\usepackage{amsfonts,amsthm,amssymb,mathrsfs,bbding}
\usepackage{txfonts}
\usepackage{graphics,multicol}
\usepackage{graphicx}
\usepackage{color}
\usepackage{caption}
\usepackage{cite}
\usepackage{latexsym,bm}
\usepackage{indentfirst}
\usepackage{xcolor}
\usepackage[colorlinks=true,
linkcolor=blue,citecolor=blue,
urlcolor=blue]{hyperref}

\evensidemargin=\oddsidemargin\topmargin=-1.5cm

\newtheorem{prob}{Problem}
\newtheorem{lem}{Lemma}[section]
\newtheorem{thm}{Theorem}[section]

\newtheorem{pro}{Proposition}[section]

\newtheorem{remark}{Remark}

\newtheorem{clm}{Claim}[section]

\theoremstyle{definition}

\addtocounter{section}{0}

\begin{document}
\title{On the spectral radius of minimally 2-(edge)-connected graphs with given size\footnote{This work is supported
by National Natural Science Foundation of China (Nos. 12061074 and 11971274) and the China Postdoctoral Science Foundation
(No. 2019M661398).}}
\author{
{\small Zhenzhen Lou$^{a}$,\ \ Gao Min$^{a}$, \ \ Qiongxiang Huang$^{a}$\footnote{Corresponding author.
Email addresses: xjdxlzz@163.com, huangqxmath@163.com.}\;}\\[2mm]
\footnotesize $^a$College of Mathematics and Systems Science, Xinjiang University,
Urumqi, Xinjiang 830046, China
}

\date{}
\maketitle {\flushleft\large\bf Abstract}
A graph is minimally $k$-connected ($k$-edge-connected) if it is $k$-connected ($k$-edge-connected) and
deleting arbitrary chosen edge always leaves a graph which is not $k$-connected ($k$-edge-connected). 
A classic result of minimally $k$-connected graph is given by Mader who determined the extremal size of a minimally
$k$-connected graph of high order in 1937.
Naturally, for a fixed size of a minimally $k$-(edge)-connected graphs,
what is the extremal spectral radius?
In this paper,
we determine the maximum spectral radius for the  minimally $2$-connected ($2$-edge-connected) graphs of given size, moreover the corresponding
extremal graphs are also determined.

\vspace{0.1cm}
\begin{flushleft}
\end{flushleft}

\section{Introduction}
A graph is said to be \emph{connected} if for every pair of vertices there is a path joining them. Otherwise the graph is
disconnected.
The \emph{connectivity} (or \emph{vertex-connectivity}) $\kappa(G)$ of a graph $G$ is the minimum number of vertices whose
removal results in a disconnected graph or in the trivial graph.  The edge-connectivity $\kappa'(G)$ is defined analogously,
only instead of vertices we remove edges.
A graph is \emph{$k$-connected} if its connectivity is at least $k$ and $k$-edge-connected if its edge-connectivity is at least
$k$.
It is almost as simple to check that
the minimal degree $\delta(G)$, the edge-connectivity and vertex connectivity satisfy the following inequality:
$$\delta(G)\ge\kappa'(G)\ge\kappa(G).$$

One of the most important task for characterization of $k$-connected graphs is to give  certain operation such that
they can be produced from simple $k$-connected graphs by repeatedly  applying  this
operation\cite{BBE}. This goal  has accomplished by Tutte \cite{Tutte} for $3$-connected graphs, by Dirac \cite{GD} and Plummer
\cite{MP} for 2-connected graphs and by Slater \cite{Slater} for $4$-connected graphs.
A graph is said to be \emph{minimally $k$-connected}  if it is $k$-connected but omitting any of edges the resulting graph is no
longer $k$-connected.
Clearly, a $k$-connected graph whose every edge
is incident with  one vertex of degree $k$ is minimally $k$-connected, especially a $k$-regular and $k$-connected graph is minimally $k$-connected.

Questions in extremal graph theory ask to maximize or minimize a graph invariant
over a fixed family of graphs.
A classic result of minimally $k$-connected graph is given by Mader who determined the extremal size of a minimally
$k$-connected graph of high order in \cite{Mader}.
Giving a graph class $\mathcal{G}$ to study the bounds of spectral radius of graphs in $\mathcal{G}$ and to characterize the extremal graphs that achieves the bound is a famous problem in the spectral extremal  graph theory \cite{BH}, which attracts some authors and have produced many interesting results  published in various magazines \cite{Nosal,VN1,Lin,Tait,Zhai}.

In the origin of researches,  $\mathcal{G}$ is restricted to the graphs of  order $n$ or size $m$. For examples,
Brualdi and Hoffman \cite{BH}  gave an upper bound on spectral radius in terms of size $m$: if $m \le \binom{k}{2}$
for some integer $k \ge 1$ then $\rho(G) \le k -1$, with equality if and only if $G$ consists of
a $k$-clique and isolated vertices. Extending this result, Stanley \cite{Stanley-1} showed that
$\rho(G) \le\frac{\sqrt{1+8m}-1}{2}$. In particular, Nosal \cite{Nosal} in 1970 proved that if $G$ is a triangle-free graph with $m$ edges then $\rho(G) \le \sqrt{m}$.

In the subsequent study, $\mathcal{G}$ is restricted to the graphs that have some combinatorial structure. For examples,
in 2002, Nikiforov \cite{VN1,VN2} showed that $\rho(G) \le\sqrt{2m(1-\frac{1}{r})}$ for a graph $G$ with given size $m$.
Bollob\'{a}s, Lee, and Letzter 2018 studied the maximizing spectral radius of subgraphs
of the hypercube for giving size $m$ \cite{Bollobas}.
Very recently, Lin, Ning and
Wu \cite{Lin} proved that $\rho(G) \le \sqrt{m-1}$ when $G$ is non-bipartite and triangle-free with giving size $m$.
Zhai, Lin and Shu \cite{Zhai} obtained that
if  $G$ contains no pentagon or hexagon with giving size $m$, then $ \rho(G) \le \frac{1}{2}+\sqrt{m-\frac{3}{4}}$, with equality holds if and only if $G$ is a book graph.

In the recent works, some authors  restrict $\mathcal{G}$  to the graphs that have connectivity. For examples,
given order $n$ of a graph,  Chen and Guo in 2019  showed that $K_{2,n-2}$ attained the maximal spectral radius among all the  minimally
2-(edge)-connected graphs \cite{CXD}.
 Fan, Goryainov and Lin in 2021 proved that $K_{3,n-3}$  has the largest spectral radius over all minimally 3-connected graphs \cite{DDF}.

All the above studies indicate that the spectral radius of a graph are related with the parameters of graphs( such as order $n$ and size $m$ ), structure of graphs ( such as forbidding subgraphs )   and vertex or edge connectivity of graphs.

Motivated by this researches,  our paper is to study the spectral extremal problem  of minimally $2$-(edge)-connected graph
under edge-condition restrictions.
Precisely, our aim is to give an answer to the following question.
\begin{prob}\label{prob-1}
Given  size $m$,
what is the maximum spectral radius among minimally $k$-(edge)-connected graphs?
\end{prob}

Denote by $SK_{2,\frac{m-1}{2}}$ the graph obtained from the complete bipartite graph $K_{2, \frac{m-1}{2}}$
by subdividing an edge once.
Set $K_{2,\frac{m-3}{2}}*K_3$  is the graph obtained by identifying a maximum degree of $K_{2,\frac{m-3}{2}}$ and a vertex of $K_3$ from the disjoint union of
$K_{2,\frac{m-3}{2}}$ and $K_3$.
\emph{A friend graph}, denoted by $F_t$, is a graph obtained from $t$ triangles by sharing a common vertex.
In this paper, we  solve Problem \ref{prob-1} for $k=2$ and attain the unique extremal graph in the following two
theorems.

\begin{thm}\label{thm-1.1}
Let $G$ be a minimally 2-connected graph of size $m$. \\
$\mathrm{(i)}$ If $m$ is even, then $\rho(G)\le \sqrt{m}$,  the equality holds if and only if $G\cong K_{2,\frac m2}$.\\
$\mathrm{(ii)}$ If $m$ is odd and $m \ge 9$,   then $\rho(G)\le \rho_1^*(m)$, where $\rho_1^*(m)$ is the largest root of $x^3-x^2-(m-2)x+m-3=0$,  the equality holds if and only if $G\cong
SK_{2,\frac{m-1}{2}}$.
\end{thm}

\begin{thm}\label{thm-1.2}
Let $G$ be a minimally 2-edge-connected graph of size $m$. \\
 $\mathrm{(i)}$ If $m$ is even, then  $\rho(G)\le \sqrt{m}$, the equality holds if and only if $G\cong K_{2,\frac m2}$.\\
 $\mathrm{(ii)}$ If  $m \ge 11$ is odd and $m \neq 15$, then  $\rho(G)\le \rho_2^*(m)$, where $\rho_2^*(m)$ is the largest root of $x^4-x^3+(1-m)x^2+(m-3)x+m-3=0$,  the equality holds if and only
 if $G \cong K_{2,\frac{m-3}{2}}*K_3$.
If $m=15$, then $\rho(G)\le \frac{1+\sqrt{41}}{2}$, the equality holds if and only if $G \cong F_5$.
\end{thm}

A famous sharp lower bound of spectral radius $\rho(G)\ge\frac{2m}{n}$ giving by L. Collatz, U. Sinogowitz in \cite{Collatz}, equality holds if and only if $G$ is a regular graph. Thus, the $m$-cycle has the minimal spectral radius among all minimally 2-(edge)-connected graphs with size $m$.
Therefore, by Theorems \ref{thm-1.1} and \ref{thm-1.2},  we obtain that
the  spectral radius of a minimally 2-(edge)-connected graph lies the interval $[2, \sqrt{m}]$.
It means that  a graph whose  spectral radius out of  $[2, \sqrt{m}]$ will not be  minimally 2-(edge)-connected which indeed indicates the relationship of spectral radius and connectivity for a graph.

 The rest of the paper is organized as follows.
In the next section, we will give some lemmas and some  properties of a minimally 2-(edge)-connected graph. In Section 3 and 4,
we will give the proof of Theorem \ref{thm-1.1} and   \ref{thm-1.2}, respectively.

\section{Preliminary}
In this section, we firstly list some symbols and then  write some properties of minimally 2-(edge)-connected graphs and
some useful lemmas.

Let $G$ be a graph with vertex set $V(G)=\{v_1,v_2,\cdots,v_n\}$ and edge set $E(G)$.
For $v \in V(G)$, $A \subset V(G)$, denote by $N(v)$ and $d(v)$ the neighborhood and the degree of the vertex $v$ in $G$,
and denote $N_A(v)=N(v) \cap A$, $d_A(u)=|N_A(v)|$. The adjacent matrix of a graph $G$ is defined as the $n \times n$ square
matrix $A(G)=(a_{ij})$ whose entries are $1$ if $v_iv_j \in E(G)$, otherwise $0$. The spectral radius of $G$, denote by
$\rho(G)$, is defined to be the largest eigenvalue of $A(G)$.
A \emph{chord} of a graph is an edge between two vertices of a cycle that is not an edge on the cycle.
If a cycle has at least one chord, then it is called a \emph{chorded cycle}.

A graph  is 2-(edge)-connected graph if it contains a 2-vertex (edge) cut set. A graph is minimally 2-(edge)-connected, introduced in \cite{BBE},
if it is a $2$-(edge)-connected  but omitting any  edge the resulting graph is no
longer $2$-(edge)-connected.
By definition, a 2-connected graph is also  2-edge-connected, but vice versa.
However there exists a minimally 2-connected graph that is not  minimally 2-edge-connected, for example the graph $H(2,2)$ shown as Fig.\ref{fig-2}. There exists a minimally 2-edge-connected graph that is not minimally 2-connected, for example the graph $C_n*C_m$. Clearly, $C_n$ is both of minimally 2-(edge)-connected. Furthermore, we will give some the properties of a minimally 2-(edge)-connected graph.

\begin{lem}[\cite{GD}]\label{le-2.1}
A minimally 2-connected graphs with more than three vertices contains no triangles.
\end{lem}

\begin{lem}[\cite{BBE}]\label{le-2.4}
Every cycle of a minimally 2-connected graph contains at least two vertices of degree two.
\end{lem}

\begin{lem}[\cite{MP}]\label{le-2.5}
$G$ is a minimally 2-connected graph if and only if no cycle of $G$ has a chord.
\end{lem}

\begin{lem}[\cite{AB}]\label{le-2.3}
If $G$ is a minimally 2-(edge)-connected graph, then $\delta(G)=2$.
\end{lem}

\begin{lem}\label{le-2.11}
A $2$-edge-connected subgraph of a minimally $2$-edge-connected graph is also minimally $2$-edge-connected.
\end{lem}
\begin{proof}
Let $G$ be a minimally 2-edge-connected graph, and $H$ be a 2-edge-connected subgraph of $G$.
By  contrary that $H$ is not minimal,
then there exists an edge $uv \in E(H)$ such that $H-uv$ is 2-edge-connected.
Since $G$ is a minimally 2-edge-connected graph,
we get that $G-uv$ is $1$-edge-connected.
Thus  $G-uv$ has a cut edge,  say $xy,$
which divides $V(G-uv)$ into  two vertices sets $U$ and $V$ such that $x\in U$ and $y\in V$,
and so $e_{G-uv}(U,V)=1$.
Also we have $u\in U$ and $v\in V$ since $G$ is a 2-edge-connected graph.
Noticed that $H-uv$ is a subgraph of $G-uv$ that is assumed to be 2-edge connected,
we claim that $H-uv$ has a cycle $C$ connecting $u$ and $v$ which must be contained in $G-uv$.
It is a contradiction.
Thus, $H$ is a minimally 2-edge-connected graph.
\end{proof}

\begin{lem}\label{le-2.6}
If $G$ is a minimally 2-edge-connected graph, then no cycle of $G$ has a chord.
\end{lem}
\begin{proof}Suppose by contrary that the cycle of the minimally 2-edge-connected graph $G$ has a chord. Then $G$ contains a
chorded cycle, which is not minimally 2-edge-connected. This is a contradiction from Lemma \ref{le-2.11}.
\end{proof}

Recall that $\kappa(G)$ ($\kappa'(G)$) denoted   the vertex connectivity (edge connectivity) of
$G$.
\begin{lem}\label{le-2.10}
If  $G$ be a minimally 2-edge-connected graph with no cut vertex,  then $G$ is minimally 2-connected.
\end{lem}
\begin{proof}
Since $G$ is a 2-edge-connected graph and $G$ has no cut vertex,  we have $2=\kappa'(G) \ge \kappa(G) \ge 2$, and so $\kappa(G)=2$.
Note that $G$ is minimally 2-edge connected.
We have $\kappa(G-e) \le \kappa'(G-e) \le 1$ for any $e \in E(G)$.
Thus $G$ is minimally 2-connected.
\end{proof}

\begin{lem}[\cite{Cvetkovi}]\label{le-2.7}
Let $G$ be a  graph with adjacency matrix $A(G)$, and let $\pi$ be an equitable partition of $G$ with quotient  matrix
$B_{\pi}$. Then $\det(x I-B_{\pi})\mid \det(x I-A(G)).$ Furthermore,  the largest
eigenvalue of $B_{\pi}$  is just the  spectral radius of $G$.
\end{lem}

By Lemma \ref{le-2.7},  we can give the bound of the spectral radius of  $SK_{2,\frac{m-1}{2}}$.
\begin{lem}\label{le-2.8}
For odd number $m>5$, we have $\rho(SK_{2,\frac{m-1}{2}})$ is the largest root of $x^3-x^2-(m-2)x+m-3=0$ and
$\sqrt{m-2}<\rho(SK_{2,\frac{m-1}{2}}))<\sqrt{m-1}$.
\end{lem}
\begin{proof}
The vertices set of $SK_{2,\frac{m-1}{2}}$ has equitable partition and the quotient matrix is
\begin{eqnarray*}
B_{\pi}=\begin{array}{ll}
\begin{pmatrix}1&1&0\\1 &0&\frac{m-3}{2}\\0&2&0\end{pmatrix}\end{array}.
\end{eqnarray*}
We have $f(x)=\det(xI_3-B_\pi)=x^3-x^2-(m-2)x+m-3$.
By Lemma \ref{le-2.7},  $\rho(SK_{2,\frac{m-1}{2}})$ is the largest root of  $f(x)=0$.
Moreover, one can verify that $f(\sqrt{m-2})<0$, and so $\rho(SK_{2,\frac{m-1}{2}})> \sqrt{m-2}$.
Also, we have
$f(\sqrt{m-1})=\sqrt{m-1}-2>0$ for $m\ge6$ and $f^\prime(x)=3x^2-2x-(m-2)>0$ for $x \ge \sqrt{m-1}$.
Thus, $\rho(SK_{2,\frac{m-1}{2}})< \sqrt{m-1}$.
\end{proof}

Notice that if each edge of a $k$-connected graph is incident with at least one vertex of degree $k$
then the graph is minimally $k$-(edge)-connected.
Clearly,  $SK_{2,\frac{m-1}{2}}$ is the minimally $2$-(edge)-connected. Moreover, we have the following lemma.
\begin{lem}\label{le-4.1}
Let $G^*$ attain the maximum spectral radius among all minimally 2-edge-connected graph of size $m\ge6$, then $\rho(G^*)>
\sqrt{m-2}$.
\end{lem}
\begin{proof}
If $m$ is even, then $\rho(G^*) \ge \rho(K_{2,\frac{m}{2}})=\sqrt{m}> \sqrt{m-2}$ since $K_{2,\frac{m}{2}}$ is minimally
2-edge-connected.  If $m$ is odd,  then
$\rho(G^*) \ge \rho(SK_{2,\frac{m-1}{2}})> \sqrt{m-2}$
from Lemma \ref{le-2.8}.
\end{proof}

\begin{lem}[\cite{AJH,QL}]\label{lem-2.6'}
Let $G$ and $H$ be two graphs, and let $P(G, x)$ be the characteristic polynomial of  $G$.\\
$\mathrm{(i)}$  If $H$ is a proper subgraph of $G$, then $\rho(H)<\rho(G)$.\\
$\mathrm{(ii)}$ If $P(H,\lambda)>P(G,\lambda)$ for $\lambda \ge \rho(G)$, then $\rho(H)<\rho(G)$.
\end{lem}

\begin{lem}[\cite{VN1,VN2}]\label{le-2.2}
Let $G$ be a $C_3$-free graph of size $m$. Then $\rho(G) \le \sqrt{m}$, the equality holds if and only if $G \cong K_{a,b}$,
where $ab=m$.
\end{lem}

\begin{lem}[\cite{Wu}]\label{le-2.9}
Let $u$, $v$ be two distinct vertices in a connected graph $G$, $\{v_i\mid i=1,2,\ldots, s\}\subseteq N_G(v)\setminus
N_G(u)$.
$X=(x_1, x_2, \ldots, x_n)^T$ is the Perron vector of $G$, where $x_i$ is corresponding to $v_i$ ($1 \leq i \leq n$).
Let $G'=G-\{vv_i\mid 1\leq i\leq s\}+\{uv_i\mid 1\leq i\leq s\}$. If $x_u\geq x_v$, then $\rho(G)<\rho(G')$.
\end{lem}

\begin{lem}[\cite{Cvetkovi}]\label{7}
Let $G$ be a connected graph with Perron vector $X=(x_1, x_2,\ldots, x_n)^T$. Let $U$, $V$ and $W$ be three disjoint subsets of $V(G)$ such that there are no edges between $V$ and $W$. Let $G'$ be the graph obtained from $G$ by deleting the edges between $V$ and $U$, and adding all the edges between $V$ and $W$. If $\sum_{u\in U}x_u\le\sum_{w\in W}x_w$, then
$\rho(G)<\rho(G')$.
\end{lem}

\begin{lem}[\cite{ZWG}]\label{le-2.12}
Let $(H, v)$ and $(K, w)$ be two connected rooted graphs. Then
$$\rho((H, v)*(K, w)) \le \sqrt{{\rho(H)}^2+{\rho(K)}^2},$$
the equality holds if and only if both $H$ and $K$ are stars, where $(H, v)*(K, w)$ is obtained by identifying $v$ and $w$
from disjoint union of $H$ and $K$.
\end{lem}

\section{Proof of Theorem \ref{thm-1.1}}
In this section,  we will give the proof of Theorem \ref{thm-1.1}.
Let $G^*$  attain  maximal spectral radius $\rho^*=\rho(G^*)$ among all minimally 2-connected graphs with size $m$.
Lemma \ref{le-2.1} and Lemma \ref{le-2.3}  indicate $G^*$ has no triangles and $\delta(G^*)=2$. First, by Lemma \ref{le-2.2} we claim that $G^*\cong K_{2, \frac m2}$ if  $m$ is even. In what follows, we always assume that $m$ is odd.

Next we will consider the structure of extremal graph $G^*$ for odd size $m\geq 9$.
Let $X=(x_1,x_2,\ldots,x_n)^T$ be the Perron vector of $G^*$ with coordinate $x_{u^*}=\max\{x_i\mid i\in V(G^*)\}$.
Denote by $A=N(u^*)$ and $B=V(G^*)\setminus{(A \cup {u^*})}$ .
Since $G^*$ has no triangles,  we have $e(A)=0$, that is,  $\sum_{i \in A}d_A(i)x_i=0$,
and thus
\begin{eqnarray}
{\rho^{*}}^{2}x_{u^*}&=&\sum\limits_{v \in A}{\sum\limits_{u\in N(v)}{x_u}}
=d(u^*)x_{u^*}+\sum_{i \in A}d_A(i)x_i+
\sum_{i \in A}d_B(i)x_i
=d(u^*)x_{u^*}+\sum_{i \in A}d_B(i)x_i\nonumber\\
&\leq&\left(d(u^*)+e(A, B)\right)x_{u^*}
=\left( m-e(B)\right)x_{u^*}
\label{eq1}
\end{eqnarray}
By Lemma \ref{le-4.1},
${\rho^*}^2> m-2$.  Combining with (\ref{eq1}),  we get $e(B)<2$.  Thus,  $e(B)=0$ or $1$.

In the following, we will give four claims to finish our proof.
\begin{clm}\label{claim-3.1}
$d(u^*)\ge3$.
\end{clm}
\begin{proof}
Otherwise, $d(u^*)\le 2$.
Since $G^*$ is minimally 2-connected,  we have $d(u^*)\ge \delta(G^*)=2$ by Lemma \ref{le-2.3}.  This induces $d(u^*)=2$
and so $|A|=2$.
We may assume $A=\{u_1, u_2 \}$.
Note that  $e(B)=0$ or $1$.
If  $e(B)=0$  then each vertex in $B$  is adjacent with both of $u_1$ and $u_2$ since $\delta(G^*)=2$.
This leads to the size of $G^*$ is even,  a contradiction.
Thus $e(B)=1$, and   we may assume $B=\{w_1w_2\}\cup I$,  where $I=\{v_1,v_2,\ldots,  v_t\}$ is  isolated vertices set.
Moreover, we see that each $v_i$ is adjacent to $u_j$ for  $j=1,2$  due to $\delta(G^*)=2$, and $N_{A}(w_1)\cap N_{A}(w_2)=\emptyset$,  $d_{A}(w_1)=d_{A}(w_2)=1$
since $G$ has no triangles.  Without loss of generality, let  $N_A(w_1)=u_1$ and $N_A(w_2)=u_2$. Now $G^*$ is determined and $m=2t+5$ in this situation, where $t\ge 2$.
By the symmetry,  we have $x_{u_1}=x_{u_2}$, $x_{w_1}=x_{w_2}$, and $x_{u^*}=x_{v_i}$ for $i=1, 2, \cdots, t$.
Thus from  $A(G^*)X=\rho^*X$, we have
 \begin{eqnarray*}\label{eq2}
 \left\{\begin{array}{ll}
 \rho^*x_{u^*}&=2x_{u_1},\\
 \rho^*x_{u_1}&=(t+1)x_{u^*}+x_{w_1},\\
 \rho^*x_{w_1}&=x_{w_2}+x_{u_1}=x_{w_1}+x_{u_1}.
 \end{array}\right.
 \end{eqnarray*}
Furthermore, we get
 \begin{equation}\label{eq3}
 ({\rho^*}^2-2t-2)x_{u^*}=\frac{1}{\rho^*-1}x_{u_1}.
 \end{equation}
Let $g(\rho^*)=({\rho^*}^2-2t-2)(\rho^*-1)-1={\rho^*}^3-{\rho^*}^2-(2t+2)\rho^*+2t+1$.
Then $g(\rho^*)={\rho^*}^3-{\rho^*}^2-(m-3)\rho^*+m-4$. Since  $\rho^*>\sqrt{m-2}$ by Lemma \ref{le-4.1} and
$$g'(\rho^*)=3{\rho^*}^2-2\rho^*-(m-3)>g'(\sqrt{m-2})=2m-2\sqrt{m-2}-3>0,$$
 we have  $g(\rho^*)>g(\sqrt{m-2})=\sqrt{m-2}-2 >0$.
Thus ${\rho^*}^2-2t-2>\frac{1}{\rho^*-1} $.
Therefore, from (\ref{eq3}) we get $x_{u^*}<x_{u_1}$,  which contradicts  the maximality of
$x_{u^*}$.
\end{proof}

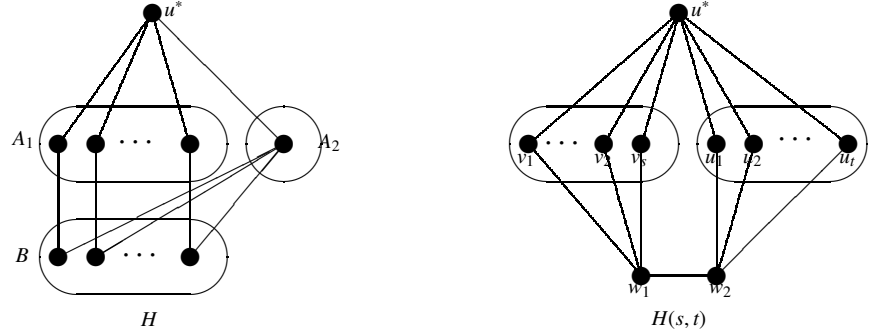
\begin{figure}[h]
\centering \setlength{\unitlength}{2.4pt}
\begin{center}
\unitlength 2.500mm 
\linethickness{0.4pt}
\ifx\plotpoint\undefined\newsavebox{\plotpoint}\fi 
\begin{picture}(44.932,17.337)(0,0)
\put(5.932,9.337){\oval(10,4)[]}
\put(5.932,3.337){\oval(10,4)[]}
\put(13.932,9.337){\oval(4,4)[]}
\put(1.932,9.337){\line(0,-1){6}}
\put(3.932,9.337){\line(0,-1){6}}
\put(8.932,9.337){\line(0,-1){6}}
\put(1.932,3.337){\line(2,1){12}}
\put(13.932,9.337){\line(-5,-3){10}}
\put(3.932,3.337){\line(0,1){0}}
\put(8.932,3.337){\line(5,6){5}}
\put(13.932,9.337){\line(-1,1){7}}
\multiput(6.932,16.337)(.03333333,-.11666667){60}{\line(0,-1){.11666667}}
\multiput(3.932,9.337)(.03370787,.07865169){89}{\line(0,1){.07865169}}
\multiput(6.932,16.337)(-.033557047,-.046979866){149}{\line(0,-1){.046979866}}
\put(30.432,9.337){\oval(9,4)[]}
\put(40.432,9.337){\oval(9,4)[]}
\multiput(26.932,9.337)(.038461538,.033653846){208}{\line(1,0){.038461538}}
\put(28.932,9.337){\line(0,1){0}}
\multiput(34.932,16.337)(-.03333333,-.11666667){60}{\line(0,-1){.11666667}}
\put(32.932,9.337){\line(0,-1){7}}
\put(32.932,2.337){\line(1,0){4}}
\put(36.932,2.337){\line(0,1){7}}
\multiput(36.932,9.337)(-.03333333,.11666667){60}{\line(0,1){.11666667}}
\multiput(34.932,16.337)(.033613445,-.058823529){119}{\line(0,-1){.058823529}}
\multiput(38.932,9.337)(-.03333333,-.11666667){60}{\line(0,-1){.11666667}}
\put(36.932,2.337){\line(1,1){7}}
\multiput(43.932,9.337)(-.043269231,.033653846){208}{\line(-1,0){.043269231}}
\multiput(32.932,2.337)(-.033707865,.039325843){178}{\line(0,1){.039325843}}
\multiput(34.932,16.337)(-.033613445,-.058823529){119}{\line(0,-1){.058823529}}
\multiput(30.932,9.337)(.03333333,-.11666667){60}{\line(0,-1){.11666667}}
\put(32.932,2.337){\line(0,1){0}}
\put(6.932,16.337){\circle*{1}}
\put(1.932,9.337){\circle*{1}}
\put(3.932,9.337){\circle*{1}}
\put(8.932,9.337){\circle*{1}}
\put(13.932,9.337){\circle*{1}}
\put(1.932,3.337){\circle*{1}}
\put(3.932,3.337){\circle*{1}}
\put(8.932,3.337){\circle*{1}}
\put(34.932,16.337){\circle*{1}}
\put(26.932,9.337){\circle*{1}}
\put(30.932,9.337){\circle*{1}}
\put(32.932,9.337){\circle*{1}}
\put(36.932,9.337){\circle*{1}}
\put(38.932,9.337){\circle*{1}}
\put(43.932,9.337){\circle*{1}}
\put(8.1,16.575){\makebox(0,0)[cc]{\scriptsize$u^*$}}
\put(36.1,16.575){\makebox(0,0)[cc]{\scriptsize$u^*$}}
\put(.074,9.588){\makebox(0,0)[cc]{\scriptsize$A_1$}}
\put(16.352,9.439){\makebox(0,0)[cc]{\scriptsize$A_2$}}
\put(0,3.419){\makebox(0,0)[cc]{\scriptsize$B$}}
\put(6.764,0){\makebox(0,0)[cc]{\scriptsize$H$}}
\put(34.933,0){\makebox(0,0)[cc]{\scriptsize$H(s,t)$}}
\put(26.832,8.5){\makebox(0,0)[cc]{\scriptsize$v_1$}}
\put(32.852,8.5){\makebox(0,0)[cc]{\scriptsize$v_s$}}
\put(30.919,8.5){\makebox(0,0)[cc]{\scriptsize$v_2$}}
\put(36.866,8.5){\makebox(0,0)[cc]{\scriptsize$u_1$}}
\put(38.872,8.5){\makebox(0,0)[cc]{\scriptsize$u_2$}}
\put(43.926,8.5){\makebox(0,0)[cc]{\scriptsize$u_t$}}
\put(32.852,1.56){\makebox(0,0)[cc]{\scriptsize$w_1$}}
\put(37.163,1.56){\makebox(0,0)[cc]{\scriptsize$w_2$}}
\put(32.932,2.337){\circle*{1}}
\put(36.932,2.337){\circle*{1}}
\put(28.838,9.216){\makebox(0,0)[cc]{$\cdots$}}
\put(41.251,9.439){\makebox(0,0)[cc]{$\cdots$}}
\put(6.169,9.291){\makebox(0,0)[cc]{$\cdots$}}
\put(6.392,3.27){\makebox(0,0)[cc]{$\cdots$}}
\end{picture}
\end{center}
\caption{\footnotesize{The graphs $H$ and $H(s,t)$, where $s,t\ge 1$.}}\label{fig-2}
\end{figure}

\begin{clm}\label{claim-3.2}
$e(B)=1$.
\end{clm}
\begin{proof}
Suppose to the contrary that $e(B)=0$,  that is, $B$ induces an independent set.
Recall that $A$ is independent,
for $u\in B$, it lies in a $4$-cycle $C=(u,v_1,u^*,v_2)$ where $v_1,v_2\in A$. There is at least one of $v_1$ and $v_2$ having degree three since otherwise $u^*$ will be a cut vertex. Also there is at least one of  $v_1$ and $v_2$ having degree two since otherwise $C$ contains at most one vertex of degree two, which contradicts Lemma \ref{le-2.4}. It implies that $d(u)=2$ due to $d(u^*)\ge 3$, and then we may assume that $d(v_1)=2$ and $d(v_2)\ge3$. In addition, the vertex as $v_2$ in $A$ is unique, because  two vertices ( in $A$ ) of degree greater than $2$ must lie in a $4$-cycle along with $u^*$ and in this cycle there have been  three vertices of degree greater than $2$. Therefore,  $G^*$ is isomorphic to $ H$ showed in Fig.\ref{fig-2}.
By quotient matrix of $H$, we have $\rho^*=1+\sqrt{\frac{m-1}{3}}< \sqrt{m-2}$ for $m \ge 9$, which contradicts Lemma \ref{le-4.1}. Thus
$e(B)=1$.
\end{proof}

\begin{clm}\label{claim-3.3}
$G^*[B]=K_2$.
\end{clm}
\begin{proof}
Otherwise,  by Claim \ref{claim-3.2}, we have $G^*[B]=\{w_1w_2\}\cup I$, where $I$ is nonempty isolated vertex set.
Choosing a vertex $v\in I$.  Clearly,  $d(v)\ge2$.
If $N(v)\subset A_1=\{v\in A\mid d(v)=2\}$, then $u^*$ is a cut vertex, a contradiction.
If $N(v)\subset A_2=A\backslash A_1$,  then $v$ is included in  a $4$-cycle which has at most one  vertex
with degree two, a contradiction.
So, $|N_{A_1}(v)|=1$ and  $|N_{A_2}(v)|=1$.
Let $v_1\in N_{A_1}(v)$, $v_2\in N_{A_2}(v)$.  Then $u^*v_1vv_2$ forms a $4$-cycle.
Notice that $|N_A(w_1)|, |N_A(w_2)|\ge1$.   Choosing two vertices  $u_1\in N(w_1)$ and $u_2\in N(w_2)$,  respectively,
then $u^*u_1w_1w_2u_2$  forms a $5$-cycle.
Moreover, the $4$-cycle and $5$-cycle either have a common vertex $u^*$ or have a common edge $u^*v_2$ ($u_2$ coincides
$v_2$).
If the former occurs then  $u^*$ is a cut vertex, otherwise  $G^{*}$ contains a  chorded cycle, a contradiction.
\end{proof}

By Claims \ref{claim-3.2} and  \ref{claim-3.3}, $G^*[B]=\{w_1w_2\}$.  Since $G^*$ has no triangles,  $N_A(w_1)\cap
N_A(w_2)=\emptyset$.
Furthermore, $N_A(w_1)\cup N_A(w_2)=A$.  Assume $\left|N_A(w_1) \right|=s$, $\left|N_A(w_2) \right|=t$. Clearly,  $s, t\ge1$
since $\delta(G^*)=2$.
Let ${H}(s,t)$ be a graph obtained from  a double star graph $D_{s,t}$ by joining an isolated vertex $u^*$ to all its leaves vertices (see Fig.\ref{fig-2}).
Now we get that $G^*\cong {H}(s,t)$ for some $s$ and $t$ satisfying $2s+2t+1=m$.

\begin{clm}\label{claim-3.4}
$G^*\cong H(1,\frac{m-3}{2})$.
\end{clm}
\begin{proof}
Without loss of generality, we assume $s\le t$.
If $s=1$, then $t=\frac{m-3}{2}$,  and so $H(s,t)\cong  H(1,\frac{m-3}{2})$. The result holds.
Suppose $G^*\cong H(s,t)$ for $s\ge2$.
Let $Y$ be the Perron eigenvector of $H(s, t)$ with spectral radius $\rho=\rho(H(s,t))$,  and let
$N_A(w_1)=\{v_1,\cdots,v_s\}$ and $N_A(w_2)=\{u_1, \cdots, u_t\}$.
By the symmetry, $y_{v_1}=\cdots=y_{v_s}$ and $y_{u_1}=\cdots=y_{u_t}$.
We have
\begin{eqnarray*}
\left\{\begin{array}{ll}\rho y_{v_1}&=y_{u^*}+y_{w_1},\\ \rho y_{u_1}&=y_{u^*}+y_{w_2},\end{array}\right.&
\left\{\begin{array}{ll}\rho y_{w_1}&=sy_{v_1}+y_{w_2},\\ \rho y_{w_2}&=ty_{u_1}+y_{w_1}.\end{array}\right.
\end{eqnarray*}
Then $y_{v_1}-y_{u_1}=\frac{1}{\rho}(y_{w_1}-y_{w_2})$ and then
$$(\rho+1)(y_{w_1}-y_{w_2})=sy_{v_1}-ty_{u_1}=\frac{s}{\rho}(y_{w_1}-y_{w_2})+(s-t)y_{u_1},$$
which indicates that
$$(\rho+1-\frac{s}{\rho})(y_{w_1}-y_{w_2})=(s-t)y_{u_1}\le0.$$
Clearly,  $\rho>\frac{s}{\rho}$ since $H(s,t)$ has  $K_{1,s}$ as a subgraph.
It follows that $y_{w_1} \le y_{w_2}$.
Note that $s\ge 2$,
obviously,  $G'=H(s,t)-v_sw_1+v_sw_2$ is also a minimally $2$-connected graph.
By Lemma \ref{le-2.9},  we have  $\rho(G')>\rho(G^*)$,  which is  a contradiction.
\end{proof}

By the definition of $SK_{2,\frac{m-1}{2}}$ and $H(1,\frac{m-3}{2})$,  it is clear that $SK_{2,\frac{m-1}{2}}\cong
H(1,\frac{m-3}{2})$. Thus
Claim \ref{claim-3.4} and Lemma \ref{le-2.8} imply the Theorem \ref{thm-1.1}.

\section{Proof of Theorem \ref{thm-1.2}}
In this section, we will give the proof of Theorem \ref{thm-1.2}.
As we know,
a \emph{block} of a graph is a maximal 2-connected subgraph with respect to  vertices.
A block of a graph is called \emph{leaf block} if it contains exactly one cut vertex.
By Lemma \ref{le-2.10},  a minimally 2-edge-connected graph $G$ without cut vertex is minimally 2-connected. Otherwise,  $G$ is made of some    blocks including at least two leaf blocks, in which each block is   minimally 2-connected by Lemma \ref{le-2.10} and they  intersect at cut vertices.
In general,  we write $G=B(t,k)$ to denote a minimally 2-edge-connected graph  with  $t$ cut vertices and $k$  blocks. If $t=0$ then $k=1$ and $G=B(0,1)$ is a type of minimally 2-connected graph that is considered in Theorem \ref{thm-1.1}. If $t>1$ then $k\ge 2$ and each block of  $G=B(t,k)$ has some cut vertices, in this case  $G=B(t,k)$ can be viewed as a tree if each block is regarded as an edge.

\begin{proof}[\bf{Proof of Theorem \ref{thm-1.2} (i)}]
We may assume that  $m\ge4$ and $X=(x_1,\ldots,x_n)^T$ is  the Perron eigenvector of $G$.

{\bf Case 1.} $G$ has no cut vertices.

In this case, $G$ is minimally 2-connected.
By Theorem \ref{thm-1.1},  $\rho(G)\le \sqrt{m}$ and  the equality holds if and only if $G\cong K_{2,\frac m2}$, which are just required.

{\bf Case 2.} $G$ has some cut vertices.

By definition,  $G=B(t, k)$ for some $t\ge 1$ and $k\ge 2$.
Let $B_1,\ldots,B_k$ be its $k$  blocks and $m(B_i)=m_i$ for $i=1, 2,\ldots, k$. We know that each $B_i$ is minimally 2-connected graph and $m=m(B(t, k))=\sum_{i=1}^k m_i$.
By Theorem \ref{thm-1.1},   $\rho(B_i) \le \sqrt{m_i}$, and the equality holds if and only if $B_i\cong K_{2,\frac
{m_i}{2}}$ for $i=1,2, \ldots, k$.
Notice that each $B_i$ is not a star since $B_i$ is minimally 2-connected. By Lemma \ref{le-2.12},  we obtain
$$\rho(B(t, k)) <\sqrt{\rho^2(B_1)+\rho^2(B_2)+\cdots+\rho^2(B_k)}\le\sqrt{m_1+m_2+\cdots+m_k}=\sqrt{m},$$
as desired.
\end{proof}
In what follows we will show (ii) of Theorem \ref{thm-1.2}, and first we give some lemmas and propositions for the preparations.
Clearly, if we transfer a leaf block of  $B(t, k)$ to another leaf block,  one can simply  verify the following result.
\begin{lem}\label{block-2}
For a minimally 2-edge-connected graph $G=B(t, k)$,  let $B_i$ be a leaf block of $G$ and $u\in B_i$
be a cut vertex. For any $v\in V(G)\setminus\{u\}$,
we have
$G'=G-\sum_{w\in N(u)\cap B_i}wu+\sum_{w\in N(u)\cap B_i}wv$ is also minimally $2$-edge-connected.
\end{lem}

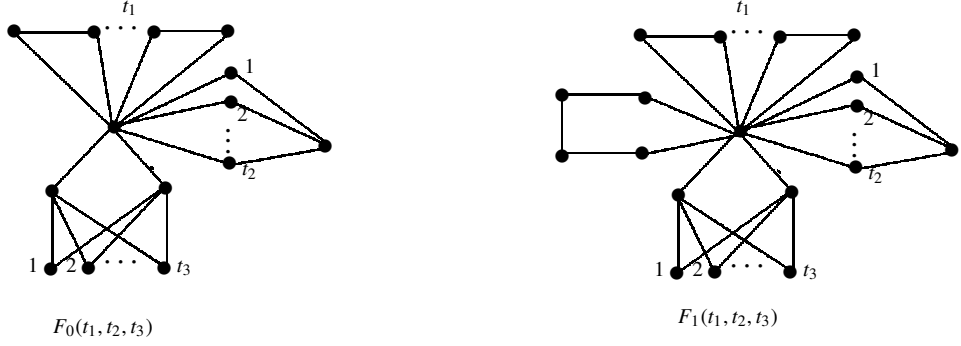
\begin{figure}[h]
\centering \setlength{\unitlength}{2.0pt}
\begin{center}
\unitlength 2.6000mm 
\linethickness{0.6pt}
\ifx\plotpoint\undefined\newsavebox{\plotpoint}\fi 
\begin{picture}(52.366,16.381)(0,0)
\multiput(41.239,9.929)(-.034655405,.033648649){148}{\line(-1,0){.034655405}}
\put(36.11,14.909){\line(1,0){4.237}}
\multiput(40.347,14.983)(.03272,-.1992){25}{\line(0,-1){.1992}}
\multiput(41.165,10.003)(.03356452,.08032258){62}{\line(0,1){.08032258}}
\put(43.246,14.983){\line(1,0){3.939}}
\multiput(47.185,14.909)(-.041232877,-.03360274){146}{\line(-1,0){.041232877}}
\multiput(41.165,9.855)(-.08032258,.03356452){62}{\line(-1,0){.08032258}}
\put(36.185,11.936){\line(-1,0){4.088}}
\put(32.097,11.862){\line(0,-1){2.899}}
\put(32.097,8.963){\line(1,0){4.088}}
\multiput(36.185,8.889)(.1654194,.0335484){31}{\line(1,0){.1654194}}
\put(41.313,9.929){\line(0,1){0}}
\put(32.023,8.963){\line(0,1){0}}
\multiput(43.097,8.071)(.0298,-.0298){5}{\line(1,0){.0298}}
\put(40.198,7.03){\line(0,1){0}}
\multiput(41.165,10.003)(.07086047,.0337093){86}{\line(1,0){.07086047}}
\multiput(47.259,12.902)(.044070796,-.033539823){113}{\line(1,0){.044070796}}
\multiput(52.239,9.112)(-.07324638,.0333913){69}{\line(-1,0){.07324638}}
\multiput(47.185,11.416)(-.156075,-.03345){40}{\line(-1,0){.156075}}
\multiput(40.942,10.078)(.10763793,-.03332759){58}{\line(1,0){.10763793}}
\multiput(47.185,8.145)(.1844444,.033037){27}{\line(1,0){.1844444}}
\put(41.685,14.983){\makebox(0,0)[cc]{$\cdots$}}
\put(36.11,14.983){\circle*{.7}}
\put(40.198,14.909){\circle*{.7}}
\put(43.246,14.909){\circle*{.7}}
\put(47.036,14.983){\circle*{.7}}
\put(32.097,11.936){\circle*{.7}}
\put(36.333,11.787){\circle*{.7}}
\put(32.097,8.814){\circle*{.7}}
\put(36.185,8.963){\circle*{.7}}
\put(47.111,8.22){\circle*{.7}}
\put(47.185,11.341){\circle*{.7}}
\put(47.185,12.828){\circle*{.7}}
\put(52.016,9.112){\circle*{.7}}
\put(48.151,13.199){\makebox(0,0)[cc]{\scriptsize$1$}}
\put(47.78,10.747){\makebox(0,0)[cc]{\scriptsize$2$}}
\put(48.151,7.822){\makebox(0,0)[cc]{\scriptsize${t_2}$}}
\put(47.036,9.781){\makebox(0,0)[cc]{$\vdots$}}
\put(41.239,10.078){\circle*{.7}}
\multiput(41.165,10.152)(-.033642105,-.0352){95}{\line(0,-1){.0352}}
\multiput(41.016,10.152)(.03361905,-.03715476){84}{\line(0,-1){.03715476}}
\multiput(43.84,7.031)(-.046453125,-.033679687){128}{\line(-1,0){.046453125}}
\put(43.915,7.031){\line(0,-1){4.162}}
\multiput(43.915,2.868)(-.051461538,.033675214){117}{\line(-1,0){.051461538}}
\put(37.894,6.808){\line(0,1){0}}
\multiput(37.969,6.808)(.03331034,-.06793103){58}{\line(0,-1){.06793103}}
\multiput(39.901,2.868)(.033508197,.035336066){122}{\line(0,1){.035336066}}
\put(38.043,6.808){\circle*{.7}}
\put(43.84,6.956){\circle*{.7}}
\put(37.969,2.794){\circle*{.7}}
\put(39.901,2.868){\circle*{.7}}
\put(43.766,2.868){\circle*{.7}}
\put(41.685,3.017){\makebox(0,0)[cc]{$\cdots$}}
\put(38.043,6.808){\line(0,-1){4.088}}
\put(44.809,2.868){\makebox(0,0)[cc]{\scriptsize$t_3$}}
\put(37.081,3.011){\makebox(0,0)[cc]{\scriptsize$1$}}
\put(39.025,3.011){\makebox(0,0)[cc]{\scriptsize$2$}}
\put(40.57,.53){\makebox(0,0)[cc]{\scriptsize$F_1(t_1,t_2,t_3)$}}
\multiput(9.216,10.137)(-.034655405,.033648649){148}{\line(-1,0){.034655405}}
\put(4.087,15.117){\line(1,0){4.237}}
\multiput(8.324,15.191)(.03272,-.1992){25}{\line(0,-1){.1992}}
\multiput(9.142,10.211)(.03356452,.08032258){62}{\line(0,1){.08032258}}
\put(11.223,15.191){\line(1,0){3.939}}
\multiput(15.162,15.117)(-.041232877,-.03360274){146}{\line(-1,0){.041232877}}
\put(9.29,10.137){\line(0,1){0}}
\put(0,9.171){\line(0,1){0}}
\multiput(11.074,8.279)(.0298,-.0298){5}{\line(0,-1){.0298}}
\put(8.175,7.238){\line(0,1){0}}
\multiput(9.142,10.211)(.07086047,.0337093){86}{\line(1,0){.07086047}}
\multiput(15.236,13.11)(.044070796,-.033539823){113}{\line(1,0){.044070796}}
\multiput(20.216,9.32)(-.07324638,.0333913){69}{\line(-1,0){.07324638}}
\multiput(15.162,11.624)(-.156075,-.03345){40}{\line(-1,0){.156075}}
\multiput(8.919,10.286)(.10763793,-.03332759){58}{\line(1,0){.10763793}}
\multiput(15.162,8.353)(.1844444,.033037){27}{\line(1,0){.1844444}}
\put(9.662,15.191){\makebox(0,0)[cc]{$\cdots$}}
\put(10.0,16.381){\makebox(0,0)[cc]{\scriptsize$t_1$}}
\put(41.39,16.381){\makebox(0,0)[cc]{\scriptsize$t_1$}}
\put(4.087,15.191){\circle*{.7}}
\put(8.175,15.117){\circle*{.7}}
\put(11.223,15.117){\circle*{.7}}
\put(15.013,15.191){\circle*{.7}}
\put(15.088,8.428){\circle*{.7}}
\put(15.162,11.549){\circle*{.7}}
\put(15.162,13.036){\circle*{.7}}
\put(19.993,9.32){\circle*{.7}}
\put(16.128,13.407){\makebox(0,0)[cc]{\scriptsize$1$}}
\put(15.757,10.955){\makebox(0,0)[cc]{\scriptsize$2$}}
\put(16.128,8.03){\makebox(0,0)[cc]{\scriptsize${t_2}$}}
\put(15.013,9.989){\makebox(0,0)[cc]{$\vdots$}}
\put(9.216,10.286){\circle*{.7}}
\multiput(9.142,10.36)(-.033642105,-.0352){95}{\line(0,-1){.0352}}
\multiput(8.993,10.36)(.03361905,-.03715476){84}{\line(0,-1){.03715476}}
\multiput(11.817,7.239)(-.046453125,-.033679687){128}{\line(-1,0){.046453125}}
\put(11.892,7.239){\line(0,-1){4.162}}
\multiput(11.892,3.076)(-.051461538,.033675214){117}{\line(-1,0){.051461538}}
\put(5.871,7.016){\line(0,1){0}}
\multiput(5.946,7.016)(.03331034,-.06793103){58}{\line(0,-1){.06793103}}
\multiput(7.878,3.076)(.033508197,.035336066){122}{\line(0,1){.035336066}}
\put(6.02,7.016){\circle*{.7}}
\put(11.817,7.164){\circle*{.7}}
\put(5.946,3.002){\circle*{.7}}
\put(7.878,3.076){\circle*{.7}}
\put(11.743,3.076){\circle*{.7}}
\put(9.662,3.225){\makebox(0,0)[cc]{$\cdots$}}
\put(6.02,7.016){\line(0,-1){4.088}}
\put(12.786,3.076){\makebox(0,0)[cc]{\scriptsize${t_3}$}}
\put(5.058,3.219){\makebox(0,0)[cc]{\scriptsize$1$}}
\put(7.002,3.219){\makebox(0,0)[cc]{\scriptsize$2$}}
\put(8.593,0){\makebox(0,0)[cc]{\scriptsize$F_0(t_1,t_2,t_3)$}}
\end{picture}
\vspace{-0.5cm}
\end{center}
\caption{\footnotesize{The graphs $F_0(t_1,t_2,t_3)$ and $F_1(t_1,t_2,t_3)$.}}\label{fig-5}
\end{figure}

Denote  by $u_1$  the maximal degree vertex of the friend graph with $t_1$ triangles,  $u_2$
a maximal degree vertex of $K_{2, t_2}$, and
 $u_3$  a vertex of $K_{2, t_3+1}$ with degree two.
Let $F_0(t_1,t_2,t_3)$ be the graph obtained from the above three graphs by identifying $u_1, u_2$ and $u_3$.
Denote $F_1(t_1,t_2,t_3)$ the graph by identifying a vertex of $C_5$ and the maximum degree vertex of $F_0(t_1,t_2,t_3)$ (see
Fig.\ref{fig-5}), where $t_i\ge0$.

In order to give the proof of Theorem \ref{thm-1.2} (ii),  we begin by proving the following two useful propositions.
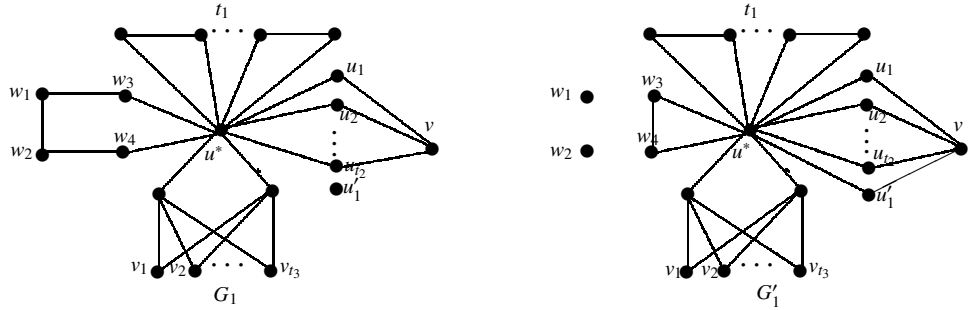
\begin{figure}[h]
\centering \setlength{\unitlength}{2.0pt}
\begin{center}
\unitlength 2.6000mm 
\linethickness{0.6pt}
\ifx\plotpoint\undefined\newsavebox{\plotpoint}\fi 
\begin{picture}(48.438,14.577)(0,0)
\multiput(10.257,8.333)(-.034655405,.033648649){148}{\line(-1,0){.034655405}}
\put(5.128,13.313){\line(1,0){4.237}}
\multiput(9.365,13.387)(.03272,-.1992){25}{\line(0,-1){.1992}}
\multiput(10.183,8.407)(.03356452,.08032258){62}{\line(0,1){.08032258}}
\put(12.264,13.387){\line(1,0){3.939}}
\multiput(16.203,13.313)(-.041232877,-.03360274){146}{\line(-1,0){.041232877}}
\multiput(10.183,8.259)(-.08032258,.03356452){62}{\line(-1,0){.08032258}}
\put(5.203,10.34){\line(-1,0){4.088}}
\put(1.115,10.266){\line(0,-1){2.899}}
\put(1.115,7.367){\line(1,0){4.088}}
\multiput(5.203,7.293)(.1654194,.0335484){31}{\line(1,0){.1654194}}
\put(10.331,8.333){\line(0,1){0}}
\put(1.041,7.367){\line(0,1){0}}
\multiput(12.115,6.475)(.0298,-.0298){5}{\line(0,-1){.0298}}
\put(9.216,5.434){\line(0,1){0}}
\multiput(10.183,8.407)(.07086047,.0337093){86}{\line(1,0){.07086047}}
\multiput(16.277,11.306)(.044070796,-.033539823){113}{\line(1,0){.044070796}}
\multiput(21.257,7.516)(-.07324638,.0333913){69}{\line(-1,0){.07324638}}
\multiput(16.203,9.82)(-.156075,-.03345){40}{\line(-1,0){.156075}}
\multiput(9.96,8.482)(.10763793,-.03332759){58}{\line(1,0){.10763793}}
\multiput(16.203,6.549)(.1844444,.033037){27}{\line(1,0){.1844444}}
\put(10.703,13.387){\makebox(0,0)[cc]{$\cdots$}}
\put(10.331,14.577){\makebox(0,0)[cc]{\scriptsize$t_1$}}
\put(5.128,13.387){\circle*{.7}}
\put(9.216,13.313){\circle*{.7}}
\put(12.264,13.313){\circle*{.7}}
\put(16.054,13.387){\circle*{.7}}
\put(1.115,10.34){\circle*{.7}}
\put(5.351,10.191){\circle*{.7}}
\put(1.115,7.218){\circle*{.7}}
\put(5.203,7.367){\circle*{.7}}
\put(16.129,6.624){\circle*{.7}}
\put(16.203,9.745){\circle*{.7}}
\put(16.203,11.232){\circle*{.7}}
\put(21.034,7.516){\circle*{.7}}
\put(5.277,10.935){\makebox(0,0)[cc]{\scriptsize$w_3$}}
\put(5.351,8.036){\makebox(0,0)[cc]{\scriptsize$w_4$}}
\put(0,10.34){\makebox(0,0)[cc]{\scriptsize$w_1$}}
\put(.074,7.293){\makebox(0,0)[cc]{\scriptsize$w_2$}}
\put(17.169,11.603){\makebox(0,0)[cc]{\scriptsize$u_1$}}
\put(16.798,9.151){\makebox(0,0)[cc]{\scriptsize$u_2$}}
\put(17.169,6.426){\makebox(0,0)[cc]{\scriptsize$u_{t_2}$}}
\put(20.885,8.63){\makebox(0,0)[cc]{\scriptsize$v$}}
\put(16.054,8.185){\makebox(0,0)[cc]{$\vdots$}}
\put(10.257,8.482){\circle*{.7}}
\multiput(10.183,8.556)(-.033642105,-.0352){95}{\line(0,-1){.0352}}
\multiput(10.034,8.556)(.03361905,-.03715476){84}{\line(0,-1){.03715476}}
\multiput(12.858,5.435)(-.046453125,-.033679687){128}{\line(-1,0){.046453125}}
\put(12.933,5.435){\line(0,-1){4.162}}
\multiput(12.933,1.272)(-.051461538,.033675214){117}{\line(-1,0){.051461538}}
\put(6.912,5.212){\line(0,1){0}}
\multiput(6.987,5.212)(.03331034,-.06793103){58}{\line(0,-1){.06793103}}
\multiput(8.919,1.272)(.033508197,.035336066){122}{\line(0,1){.035336066}}
\put(7.061,5.212){\circle*{.7}}
\put(12.858,5.36){\circle*{.7}}
\put(6.987,1.198){\circle*{.7}}
\put(8.919,1.272){\circle*{.7}}
\put(12.784,1.272){\circle*{.7}}
\put(10.703,1.421){\makebox(0,0)[cc]{$\cdots$}}
\put(7.061,5.212){\line(0,-1){4.088}}
\put(13.827,1.272){\makebox(0,0)[cc]{\scriptsize$v_{t_3}$}}
\multiput(37.311,8.333)(-.034655405,.033648649){148}{\line(-1,0){.034655405}}
\put(32.182,13.313){\line(1,0){4.237}}
\multiput(36.419,13.387)(.03272,-.1992){25}{\line(0,-1){.1992}}
\multiput(37.237,8.407)(.03356452,.08032258){62}{\line(0,1){.08032258}}
\put(39.318,13.387){\line(1,0){3.939}}
\multiput(43.257,13.313)(-.041232877,-.03360274){146}{\line(-1,0){.041232877}}
\multiput(37.237,8.259)(-.08032258,.03356452){62}{\line(-1,0){.08032258}}
\multiput(32.257,7.293)(.1654194,.0335484){31}{\line(1,0){.1654194}}
\put(37.385,8.333){\line(0,1){0}}
\put(28.095,7.367){\line(0,1){0}}
\multiput(39.169,6.475)(.0298,-.0298){5}{\line(0,-1){.0298}}
\put(36.27,5.434){\line(0,1){0}}
\multiput(37.237,8.407)(.07086047,.0337093){86}{\line(1,0){.07086047}}
\multiput(43.331,11.306)(.044070796,-.033539823){113}{\line(1,0){.044070796}}
\multiput(48.311,7.516)(-.07324638,.0333913){69}{\line(-1,0){.07324638}}
\multiput(43.257,9.82)(-.156075,-.03345){40}{\line(-1,0){.156075}}
\put(37.757,13.387){\makebox(0,0)[cc]{$\cdots$}}
\put(37.385,14.577){\makebox(0,0)[cc]{\scriptsize$t_1$}}
\put(32.182,13.387){\circle*{.7}}
\put(36.27,13.313){\circle*{.7}}
\put(39.318,13.313){\circle*{.7}}
\put(43.108,13.387){\circle*{.7}}
\put(28.969,10.14){\circle*{.7}}
\put(32.405,10.191){\circle*{.7}}
\put(28.969,7.418){\circle*{.7}}
\put(32.257,7.367){\circle*{.7}}
\put(43.257,9.745){\circle*{.7}}
\put(43.257,11.232){\circle*{.7}}
\put(48.088,7.516){\circle*{.7}}
\put(32.331,10.935){\makebox(0,0)[cc]{\scriptsize$w_3$}}
\put(32.105,8.036){\makebox(0,0)[cc]{\scriptsize$w_4$}}
\put(27.654,10.34){\makebox(0,0)[cc]{\scriptsize$w_1$}}
\put(27.654,7.293){\makebox(0,0)[cc]{\scriptsize$w_2$}}
\put(44.223,11.603){\makebox(0,0)[cc]{\scriptsize$u_1$}}
\put(43.852,9.151){\makebox(0,0)[cc]{\scriptsize$u_2$}}
\put(47.939,8.63){\makebox(0,0)[cc]{\scriptsize$v$}}
\put(37.311,8.482){\circle*{.7}}
\multiput(37.237,8.556)(-.033642105,-.0352){95}{\line(0,-1){.0352}}
\multiput(37.088,8.556)(.03361905,-.03715476){84}{\line(0,-1){.03715476}}
\multiput(39.912,5.435)(-.046453125,-.033679687){128}{\line(-1,0){.046453125}}
\put(39.987,5.435){\line(0,-1){4.162}}
\multiput(39.987,1.272)(-.051461538,.033675214){117}{\line(-1,0){.051461538}}
\put(33.966,5.212){\line(0,1){0}}
\multiput(34.041,5.212)(.03331034,-.06793103){58}{\line(0,-1){.06793103}}
\multiput(35.973,1.272)(.033508197,.035336066){122}{\line(0,1){.035336066}}
\put(34.115,5.212){\circle*{.7}}
\put(39.912,5.36){\circle*{.7}}
\put(34.041,1.198){\circle*{.7}}
\put(35.973,1.272){\circle*{.7}}
\put(39.838,1.272){\circle*{.7}}
\put(37.757,1.421){\makebox(0,0)[cc]{$\cdots$}}
\put(34.115,5.212){\line(0,-1){4.088}}
\put(40.881,1.272){\makebox(0,0)[cc]{\scriptsize$v_{t_3}$}}
\put(16.129,5.435){\circle*{.7}}
\put(16.99,5.435){\makebox(0,0)[cc]{\scriptsize$u_1'$}}
\put(32.332,10.191){\line(0,-1){3.122}}
\put(36.865,7.441){\makebox(0,0)[cc]{\scriptsize$u^*$}}
\put(9.885,7.293){\makebox(0,0)[cc]{\scriptsize$u^*$}}
\put(33.411,1.238){\makebox(0,0)[cc]{\scriptsize$v_1$}}
\put(35.267,1.215){\makebox(0,0)[cc]{\scriptsize$v_2$}}
\put(6.099,1.415){\makebox(0,0)[cc]{\scriptsize$v_1$}}
\put(8.043,1.315){\makebox(0,0)[cc]{\scriptsize$v_2$}}
\put(10.518,0){\makebox(0,0)[cc]{\scriptsize$G_1$}}
\put(38.272,0){\makebox(0,0)[cc]{\scriptsize$G_1'$}}
\multiput(37.262,8.567)(.09995082,-.03359016){61}{\line(1,0){.09995082}}
\multiput(43.359,6.518)(.1494375,.0328438){32}{\line(1,0){.1494375}}
\put(48.141,7.569){\line(-2,-1){4.73}}
\multiput(43.411,5.204)(-.0656875,.033395833){96}{\line(-1,0){.0656875}}
\put(43.359,6.518){\circle*{.7}}
\put(43.359,5.151){\circle*{.7}}
\put(43.254,8.305){\makebox(0,0)[cc]{$\vdots$}}
\put(44.2,7.043){\makebox(0,0)[cc]{\scriptsize$u_{t_2}$}}
\put(44.252,5.046){\makebox(0,0)[cc]{\scriptsize$u_1'$}}
\end{picture}
\vspace{-0.5cm}
\end{center}
\caption{\footnotesize{The vertex labels of $G_1\cong F_1(t_1,t_2,t_3)\cup K_1$ and $G_1'\cong F_0(t_1+1,t_2+1,t_3)\cup 2K_1$.}}\label{fig-6}
\end{figure}

\begin{pro}\label{claim-4.9}
$\rho(F_1(t_1,t_2,t_3))<\rho(F_0(t_1+1,t_2+1,t_3))$, where $F_1(t_1,t_2,t_3)$
and $F_0(t_1+1,t_2+1,t_3)$ have the same number of edges
\begin{equation}\label{m-t}
m=
\left\{
\begin{array}{ll}
3t_1+2t_2+2t_3+7&\ \ \mbox{for any $t_1\ge0$, $t_2=0$ or $\ge2$ and $t_3\ge1$},\\
3t_1+2t_2+5&\ \ \mbox{for any $t_1\ge0$, $t_2=0$ or $\ge2$ and $t_3=0$.}
\end{array}\right.
\end{equation}
\end{pro}
\begin{proof}
Let $G_1=F_1(t_1,t_2,t_3)\cup K_1$
and let
$$G_1'=G_1-(w_1w_2+w_1w_3+w_2w_4)+(w_3w_4+u^*u_1'+u_1'v)\cong F_1(t_1+1,t_2+1,t_3)\cup 2K_1, $$
where the  labels of $V(G_1)$ and $V(G_1')$  are shown in  Fig.\ref{fig-7}.
Clearly,  $\rho(G_1)=\rho(F_1(t_1,t_2,t_3))$, $\rho(G_1')=\rho(F_0(t_1+1,t_2+1,t_3))$ and
$$m(G_1)=m(G_1')=m=
\left\{
\begin{array}{ll}
3t_1+2t_2+2t_3+7&\ \ \mbox{for any $t_1\ge0$, $t_2=0$ or $\ge2$ and $t_3\ge1$},\\
3t_1+2t_2+5&\ \ \mbox{for any $t_1\ge0$, $t_2=0$ or $\ge2$ and $t_3=0$.}
\end{array}\right.
$$
It suffices to show $\rho=\rho(G_1)< \rho'=\rho(G_1')$.

Let $Y=(y_1,\ldots, y_n)$ and $Z=(z_1,\ldots, z_n)$ be the Perron eigenvector of $G_1$ and
$G_1'$, respectively.
Then we obtain
\begin{eqnarray*}
(\rho'-\rho)Y^TZ&=&Y^TA(G_1')Z-Y^TA(G_1)Z\\
&=&\sum_{ij \in E(G_1')}(y_iz_j+z_iy_j)-\sum_{ij \in
E(G_1)}(y_iz_j+z_iy_j)\\
&=&(y_{w_3}z_{w_4}+z_{w_3}y_{w_4})+(y_{u^*}z_{u_1'}+z_{u^*}y_{u_1'})
+(y_{u_1'}z_{v}+z_{u_1'}y_{v})\\
& &-[(y_{w_1}z_{w_3}+z_{w_1}y_{w_3})+(y_{w_2}z_{w_1}+z_{w_2}y_{w_1})+(y_{w_2}z_{w_4}+z_{w_2}y_{w_4})].
\end{eqnarray*}
Notice that  $y_{u_1'}=0$, $y_{w_1}=y_{w_2}$, $y_{w_3}=y_{w_4}$, $z_{w_1}=z_{w_2}=0$ and $z_{w_3}=z_{w_4}$, we have
\begin{eqnarray}\label{eq1'}
(\rho'-\rho)Y^TZ=
2(y_{w_3}-y_{w_1})z_{w_3}+y_{u^*}z_{u_1'}+y_{v}z_{u_1'}.
\end{eqnarray}
Since $C_5$ is a proper subgraph of $G_1$,  by Lemma
\ref{lem-2.6'}, we have
$\rho > \rho(C_5)=2$. By the eigen-equation $\rho y_{w_1}=y_{w_2}+y_{w_3}=y_{w_1}+y_{w_3}$, we have $y_{w_3}=(\rho-1)y_{w_1}>y_{w_1}$, and so the right of (\ref{eq1'}) is more than $0$.
Note that $Y^TZ\ge0$.
It follows that $\rho'>\rho$.
\end{proof}

\begin{figure}[h]
\centering \setlength{\unitlength}{2.0pt}
\begin{center}
\unitlength 2.5000mm 
\linethickness{0.6pt}
\ifx\plotpoint\undefined\newsavebox{\plotpoint}\fi 
\begin{picture}(48.787,16.136)(0,0)
\multiput(10.434,10.453)(-.034655405,.033648649){148}{\line(-1,0){.034655405}}
\put(5.305,15.433){\line(1,0){4.237}}
\multiput(9.542,15.507)(.03272,-.1992){25}{\line(0,-1){.1992}}
\multiput(10.36,10.527)(.03356452,.08032258){62}{\line(0,1){.08032258}}
\put(12.441,15.507){\line(1,0){3.939}}
\multiput(16.38,15.433)(-.041232877,-.03360274){146}{\line(-1,0){.041232877}}
\put(10.508,10.453){\line(0,1){0}}
\multiput(12.292,8.595)(.0298,-.0298){5}{\line(0,-1){.0298}}
\put(10.88,16.5){\makebox(0,0)[cc]{\scriptsize$t_1$}}
\put(40.25,16.5){\makebox(0,0)[cc]{\scriptsize$t_1$}}
\put(9.393,7.554){\line(0,1){0}}
\put(10.88,15.507){\makebox(0,0)[cc]{$\cdots$}}
\put(5.305,15.507){\circle*{.7}}
\put(9.393,15.433){\circle*{.7}}
\put(12.441,15.433){\circle*{.7}}
\put(16.231,15.507){\circle*{.7}}
\put(10.434,10.602){\circle*{.7}}
\put(7.089,7.332){\line(0,1){0}}
\multiput(10.341,10.518)(-.07258929,.03314286){56}{\line(-1,0){.07258929}}
\put(6.276,12.374){\line(-1,0){5.038}}
\put(6.276,10.165){\line(0,1){.088}}
\multiput(1.237,12.374)(.033673469,-.034272109){147}{\line(0,-1){.034272109}}
\multiput(6.187,7.336)(-.05401111,.03338889){90}{\line(-1,0){.05401111}}
\multiput(1.326,10.341)(.08259016,.03332787){61}{\line(1,0){.08259016}}
\put(6.364,12.374){\line(-1,-1){5.127}}
\put(1.237,7.248){\line(1,0){5.038}}
\put(6.276,12.463){\circle*{.7}}
\put(6.187,7.159){\circle*{.7}}
\put(1.237,12.374){\circle*{.7}}
\put(1.326,10.341){\circle*{.7}}
\put(1.326,7.248){\circle*{.7}}
\put(1.326,9.016){\makebox(0,0)[cc]{$\vdots$}}
\put(6.276,11.137){\makebox(0,0)[cc]{\scriptsize$v_1'$}}
\put(6.187,8.309){\makebox(0,0)[cc]{\scriptsize$v_2'$}}
\put(0,12.905){\makebox(0,0)[cc]{\scriptsize$v_1$}}
\put(.088,10.253){\makebox(0,0)[cc]{\scriptsize$v_2$}}
\put(.265,7.336){\makebox(0,0)[cc]{\scriptsize$v_{t_3}$}}
\multiput(6.099,7.336)(.047452632,.033494737){95}{\line(1,0){.047452632}}
\multiput(10.518,10.695)(-.033494737,-.047452632){95}{\line(0,-1){.047452632}}
\multiput(7.336,6.187)(.033637168,-.04380531){113}{\line(0,-1){.04380531}}
\multiput(11.137,1.237)(-.03314286,.093125){56}{\line(0,1){.093125}}
\multiput(9.281,6.452)(.033125,.1353437){32}{\line(0,1){.1353437}}
\multiput(10.341,10.783)(.033711864,-.037449153){118}{\line(0,-1){.037449153}}
\multiput(14.319,6.364)(-.033630435,-.055728261){92}{\line(0,-1){.055728261}}
\put(7.336,6.276){\circle*{.7}}
\put(9.369,6.364){\circle*{.7}}
\put(14.231,6.364){\circle*{.7}}
\put(11.225,1.237){\circle*{.7}}
\put(11.667,6.364){\makebox(0,0)[cc]{$\cdots$}}
\put(15.468,6.364){\circle*{.7}}
\put(19.269,6.364){\circle*{.7}}
\put(17.324,6.276){\makebox(0,0)[cc]{$\cdots$}}
\put(7.248,5.392){\makebox(0,0)[cc]{\scriptsize$u_1$}}
\put(9.988,5.568){\makebox(0,0)[cc]{\scriptsize$u_2$}}
\put(14.0,5.568){\makebox(0,0)[cc]{\scriptsize$u_{t_2}$}}
\put(15.291,7.248){\makebox(0,0)[cc]{\scriptsize$u_1'$}}
\put(19.092,7.259){\makebox(0,0)[cc]{\scriptsize$u'_{t_3+1}$}}
\put(11.5,10.5){\makebox(0,0)[cc]{\scriptsize$u^*$}}
\put(12.286,1.414){\makebox(0,0)[cc]{\scriptsize$v$}}
\multiput(39.602,10.732)(-.034655405,.033648649){148}{\line(-1,0){.034655405}}
\put(34.473,15.712){\line(1,0){4.237}}
\multiput(38.71,15.786)(.03272,-.1992){25}{\line(0,-1){.1992}}
\multiput(39.528,10.806)(.03356452,.08032258){62}{\line(0,1){.08032258}}
\put(41.609,15.786){\line(1,0){3.939}}
\multiput(45.548,15.712)(-.041232877,-.03360274){146}{\line(-1,0){.041232877}}
\put(39.676,10.732){\line(0,1){0}}
\multiput(41.46,8.874)(.0298,-.0296){5}{\line(1,0){.0298}}
\put(38.561,7.833){\line(0,1){0}}
\put(40.048,15.786){\makebox(0,0)[cc]{$\cdots$}}
\put(34.473,15.786){\circle*{.7}}
\put(38.561,15.712){\circle*{.7}}
\put(41.609,15.712){\circle*{.7}}
\put(45.399,15.786){\circle*{.7}}
\put(39.602,10.881){\circle*{.7}}
\put(36.257,7.611){\line(0,1){0}}
\put(35.444,12.742){\circle*{.7}}
\put(35.355,7.438){\circle*{.7}}
\put(30.405,12.652){\circle*{.7}}
\put(30.494,10.62){\circle*{.7}}
\put(30.494,7.526){\circle*{.7}}
\put(30.494,9.294){\makebox(0,0)[cc]{\scriptsize$\vdots$}}
\put(35.444,11.615){\makebox(0,0)[cc]{\scriptsize$v_1'$}}
\put(35.355,8.587){\makebox(0,0)[cc]{\scriptsize$v_2'$}}
\put(29.168,13.184){\makebox(0,0)[cc]{\scriptsize$v_1$}}
\put(29.256,10.531){\makebox(0,0)[cc]{\scriptsize$v_2$}}
\put(29.433,7.614){\makebox(0,0)[cc]{\scriptsize$v_{t_3}$}}
\multiput(39.686,10.973)(-.033494737,-.047452632){95}{\line(0,-1){.047452632}}
\multiput(36.504,6.465)(.033637168,-.04379646){113}{\line(0,-1){.04379646}}
\multiput(40.305,1.516)(-.03314286,.09310714){56}{\line(0,1){.09310714}}
\multiput(38.449,6.73)(.033125,.135375){32}{\line(0,1){.135375}}
\multiput(39.509,11.062)(.033711864,-.037457627){118}{\line(0,-1){.037457627}}
\multiput(43.487,6.642)(-.033630435,-.055717391){92}{\line(0,-1){.055717391}}
\put(36.504,6.554){\circle*{.7}}
\put(38.537,6.642){\circle*{.7}}
\put(43.399,6.642){\circle*{.7}}
\put(40.393,1.516){\circle*{.7}}
\put(40.835,6.642){\makebox(0,0)[cc]{$\cdots$}}
\put(44.636,6.642){\circle*{.7}}
\put(48.437,6.642){\circle*{.7}}
\put(46.492,6.554){\makebox(0,0)[cc]{$\cdots$}}
\put(36.416,5.67){\makebox(0,0)[cc]{\scriptsize$u_1$}}
\put(39.156,5.847){\makebox(0,0)[cc]{\scriptsize$u_2$}}
\put(43.445,5.847){\makebox(0,0)[cc]{\scriptsize$u_{t_2}$}}
\put(44.459,7.526){\makebox(0,0)[cc]{\scriptsize$u_1'$}}
\put(48.26,7.538){\makebox(0,0)[cc]{\scriptsize$u'_{t_3+1}$}}
\put(40.92,11){\makebox(0,0)[cc]{\scriptsize$u^*$}}
\put(41.454,1.693){\makebox(0,0)[cc]{\scriptsize$v$}}
\multiput(39.686,11.049)(.039992063,-.033674603){126}{\line(1,0){.039992063}}
\multiput(44.725,6.806)(-.033581395,-.04179845){129}{\line(0,-1){.04179845}}
\multiput(40.393,1.414)(.050275,.0337){160}{\line(1,0){.050275}}
\multiput(48.43,6.806)(-.069891473,.033573643){129}{\line(-1,0){.069891473}}
\put(10.164,0){\makebox(0,0)[cc]{\scriptsize$G_2$}}
\put(37.918,0){\makebox(0,0)[cc]{\scriptsize$G_2'$}}
\end{picture}

\vspace{-0.5cm}
\end{center}
\caption{\footnotesize{The  vertex labels of
$G_2\cong F_0(t_1,t_2,t_3)\cup(t_3+1)K_1$  and  $G_2'\cong F_0(t_1,t_2+t_3+1,0)\cup(t_3+2)K_1$.}}\label{fig-7}
\end{figure}
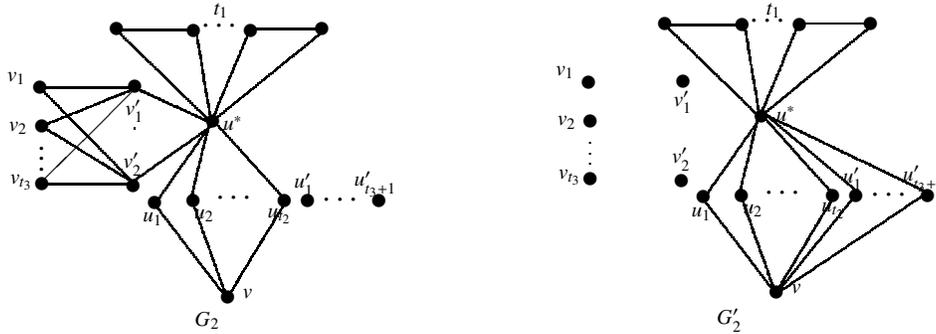

\begin{pro}\label{claim-4.9'}
$\rho(F_0(t_1,t_2,t_3))<\rho(F_0(t_1,t_2+t_3+1,0))$,
where $F_0(t_1,t_2,t_3)$ and $F_0(t_1,t_2+t_3+1,0)$ have the same number of edges $m=3t_1+2t_2+2t_3+2$, where $t_1\ge0$, $t_2=0$ or $\ge2$, and $t_3\ge1$.
\end{pro}
\begin{proof}
Let $G_2=F_0(t_1,t_2,t_3)\cup (t_3+1)K_1$, its vertices
be labelled
  as in Fig.\ref{fig-6} and the isolated set
$I_2=\{u_1',u_2', \cdots, u_{t_3+1}'\}$.
Let
$$G_2'=G_2-(\sum_{i=1}^{2}u^*v_i'+\sum_{i=1}^{2}\sum_{j=1}^{t_3}v_i'v_j)+(\sum_{i=1}^{t_3+1}(u^*u_i'+vu_i'))\cong
 F_0(t_1,t_2+t_3+1,0)\cup (t_3+2)K_1.$$
Clearly, $\rho(F_0(t_1,t_2,t_3))=\rho(G_2)$ and
$\rho( F_0(t_1,t_2+t_3+1,0))=\rho(G_2')$. Also, $F_0(t_1,t_2,t_3)$ and $F_0(t_1,t_2+t_3+1,0)$ have the same number of edges $3t_1+2t_2+2t_3+2=m$. 
It suffices to show $\rho=\rho(G_2)< {\rho}'=\rho(G_2')$.

 Suppose to the contrary that $\rho\ge{\rho}'$.
Let
$Y=(y_1,\ldots, y_n)$ and $Z=(z_1,\ldots, z_n)$ be the Perron eigenvector of $G_2$ and
$G_2'$, respectively.
By symmetry, we have $y_{u_1'}=0$, and $z_{v_1}=z_{v_1'}=0$.
Thus we obtain
\begin{eqnarray}\label{eq3'}
&&(\rho{'}-\rho)Y^TZ=Y^TA(G_2^{'})Z-Y^TA(G_2)Z
=\sum_{ij \in E(G_2^{'})}(y_iz_j+z_iy_j)-\sum_{ij \in
E(G_2)}(y_iz_j+z_iy_j) \nonumber\\
&=&(t_3+1)(y_{u^*}z_{u_1'}+z_{u^*}y_{u_1'}+y_{v}z_{u_1'}+z_{v}y_{u_1'})
-2(y_{u^*}z_{v_1'}+z_{u^*}y_{v_1'})-2t_3(y_{v_1}z_{v_1'}+z_{v_1}y_{v_1'})\nonumber\\
&=&(t_3+1)(y_{u^*}z_{u_1'}+y_{v}z_{u_1'})-2y_{v_1'}z_{u^*}.
\end{eqnarray}
By eign-equation of $A(G_2)$ and $A(G'_2)$, we have
\begin{eqnarray*}
\left\{\begin{array}{ll} \rho y_{v}=t_2y_{u_1},\\ \rho y_{u_1}=y_{u^*}+y_v.
\end{array}\right.&
\left\{\begin{array}{ll} \rho y_{v_1'}=y_{u^*}+t_3y_{v_1},\\ \rho y_{v_1}=2y_{v_1'}.
\end{array}\right.&
\left\{\begin{array}{ll} \rho' z_{u_1'}=z_{u^*}+z_v,\\
\rho'z_v=(t_3+t_2+1)z_{u_1'}. \end{array}\right.
\end{eqnarray*}
Then we obtain
$$y_v=\frac{t_2}{{\rho}^2-t_2}y_{u^*},\ \
y_{v_1'}=\frac{\rho}{{\rho}^2-2t_3}y_{u^*}\ \
\mbox{and} \ \ z_{u_1'}=\frac{\rho'}{{\rho'}^2-(t_3+t_2+1)}z_{u^*}.
$$
From (\ref{eq3'}),  we get
\begin{eqnarray}\label{eq4'}
(\rho'-\rho)Y^TZ&=&(\frac{(t_3+1)\rho'}{{\rho'}^2-(t_3+t_2+1)}\frac{{\rho}^2}{{\rho}^2-t_2}-\frac{2\rho}{{\rho}^2-2t_3})y_{u^*}z_{u^*}.
\end{eqnarray}
Since $\frac{(t_3+1)\rho'}{{\rho'}^2-(t_3+t_2+1)}$ monotonically decreases with respect to $\rho'$ and
$\rho \ge \rho'$, from (\ref{eq4'}) we  get
\begin{eqnarray}\label{eq5'}
(\rho'-\rho)Y^TZ&\ge&(\frac{(t_3+1)\rho}{{\rho}^2-(t_3+t_2+1)}\frac{{\rho}^2}{{\rho}^2-t_2}-\frac{2\rho}{{\rho}^2-2t_3})y_{u^*}z_{u^*}\nonumber\\
&=&(\frac{t_3+1}{{\rho}^2-(t_3+t_2+1)}\frac{{\rho}^2}{{\rho}^2-t_2}-\frac{2}{{\rho}^2-2t_3})\rho
y_{u^*}z_{u^*}\nonumber\\
&\ge&(\frac{t_3+1}{{\rho}^2-(t_3+t_2+1)}-\frac{2}{{\rho}^2-2t_3})\rho y_{u^*}z_{u^*}
\end{eqnarray}
Let $g(\rho)=(t_3+1)({\rho}^2-2t_3)-2({\rho}^2-(t_3+t_2+1))
=(t_3-1){\rho}^2-2t_3^2+2t_2+2$. It is clear
that $g(\rho) >0$  for $\rho > \rho(K_{2,t_3+1})=\sqrt{2t_3+2}$.
One can also verify that $\rho^2>t_2+t_3+1$.
Hence,
$$\frac{(t_3+1)}{{\rho}^2-(t_3+t_2+1)}-\frac{2}{{\rho}^2-2t_3}>0.$$
From (\ref{eq5'})
we have  $\rho'>\rho$, a   contradiction. Therefore, $\rho< \rho'$.
\end{proof}

Now is the time to prove (ii) of Theorem \ref{thm-1.2}.

\begin{proof}[\bf{Proof of Theorem \ref{thm-1.2} (ii)}]
Let $G^*$ be the graph  with the maximum spectral radius over all minimally 2-edge-connected graph of odd size $m \ge 11$,  and let
$X=(x_1,\ldots,x_n)^T$
be the Perron eigenvector of $G^*$ with coordinate $x_{u^*}=\max\{x_i\mid i\in V(G^*)\}$. Denote by $\rho^*=\rho(G^*)$,
$A=N(u^*)$ and $B=V(G^*)\setminus{(A \cup {u^*})}$. Notice that $\delta(G^*)=2$.
Now we give  Claims \ref{claim-4.1}-\ref{claim-4.8} to finish the proof of Theorem \ref{thm-1.2} (ii).

\begin{clm}\label{claim-4.1}
 $G^*[A]$ is isomorphic  to the union of some independent edges and isolated vertices.
\end{clm}
\begin{proof}
On the one hand, $G^*[A]$ contains no cycle. Otherwise, we assume that a cycle $C_l \subset G^*[A]$ ($l\ge3$).
Then there exists a wheel $W_{l+1}$ in $G^*$,  it  forms a chorded cycle  in $G^*$,  which contradicts   Lemma
\ref{le-2.6}.
On the other hand, $G^*[A]$ contains no $P_3$.  Otherwise, $G^*$ contains  a chorded cycle with order $4$, a contradiction.
\end{proof}
Let $A_1$ be the isolated vertex set of $G^*[A]$. Then  $A_2=A \setminus A_1$ consists of some independent edges if $A_2\not=\emptyset$.
\begin{clm}\label{claim-4.2}
$N_B(u)= \emptyset$ for any  $u \in A_2$.
\end{clm}
\begin{proof}
Otherwise, there exists a vertex $v \in B$ that is adjacent to a vertex  $u_2 \in A_2$. We may further assume  that  $u_2\sim u_2'\in A_2$. If $v$ has no neighbor in $B$, then
there exists a vertex $u \in A$ adjacent to $v$ due to $\delta(G^*)\ge 2$. It follows that
$$\left\{\begin{array}{ll}
C=u^*u_2'u_2vu_1 \mbox{ is a cycle with the chord $u^*u_2$ }& \mbox{ if $u=u_1\in A_1$ }\\
C=u^*u_2'vu_2 \mbox{ is a cycle with the chord $u_2'u_2$ }& \mbox{ if $u=u_2'\in A_2$ }\\
C=u^*u_2'u_2vu_3 \mbox{ is a cycle with the chord $u_2^*u_2$ }& \mbox{ if $u=u_3\in A_2$ }
\end{array}\right.
$$
It is impossible since any cycle  of  $G^*$ has no chord.  So,  $d_B(v)
\ge1$. However, in this situation, there exists a path $P:=vv_1 \cdots v_t$ in $G^*[B]$ such that $v_t$ is adjacent to some $u'\in A$ since otherwise $u_2v$ will be a cut edge. By regarding   $u'$ as the above $u$, as similar above  we can  find a chorded cycle in $G^*$, a contradiction.
\end{proof}
If $A_2 \neq \emptyset$, then, from Claim \ref{claim-4.1}, $\{u^*\}\cup A_2$ induces $t_1=\frac{|A_2 |}{2}$'s triangles with a common vertex $u^*$. Moreover, we see from Claim \ref{claim-4.2} that each of these triangles must be a leaf block of $G^*$.

\begin{clm}\label{claim-4.2'}  $e(B)=0$ or $1$.\end{clm}
\begin{proof}
By Claims \ref{claim-4.1}  and \ref{claim-4.2}, we know that  $A_2$ induces some independent edges and $N_B(A_2)=\emptyset$.
Recall that $A_1$ induces some isolated vertices.
Clearly,  $\sum_{i \in A_1}d_{A_1}(i)x_i=0$.
From Lemma \ref{le-4.1}, $\rho^*>\sqrt{m-2}\ge3$ for $m\ge11$.
By symmetry,  for any $i, j \in A_2$, we have $x_i=x_j$
and so $\rho^*x_i=x_{u^*}+x_j= x_{u^*}+x_i$,
which induces $x_i= \frac{x_{u^*}}{\rho^*-1}<\frac{x_{u^*}}{2}$.
Then we have
\begin{eqnarray*}
{\rho^*}^2x_{u^*}&=&d(u^*)x_{u^*}+\sum_{i \in A_1}d_{A_1}(i)x_i+\sum_{i \in A_2}d_{A_2}(i)x_i+\sum_{i \in B}d_A(i)x_i\\
&<&d(u^*)x_{u^*}+\frac{x_{u^*}}{2}\sum_{i \in A_2}d_{A_2}(i)+e(A,B)x_{u^*}\\
&=&d(u^*)x_{u^*}+\frac{x_{u^*}}{2}\cdot2e(A_2)+e(A,B)x_{u^*}\\
 &= &(d(u^*)+e(A)+e(A,B))x_{u^*}\\
 &=&(m-e(B))x_{u^*}.
\end{eqnarray*}
Combining it with $(m-2)x_{u^*}<{\rho^*}^2x_{u^*}$,  we have $e(B)<2$. It follows the result.
\end{proof}

If $e(B)=1$, we may denote $e=w_1^*w_2^*$ the unique edge in $G^*[B]$ in what follows.
Without loss of generality, we may assume $d_A(w_1^*)\le d_A(w_2^*)$.

\begin{clm}\label{claim-4.3}
$G^*[\{u^*, w_1^*, w_2^*\}\cup N_{A_1}(\{w_1^*, w_2^*\})] \cong C_5$.
 \end{clm}
\begin{proof}
Firstly, we will show $N_{A_1}(w_1^*)\cap N_{A_1}(w_2^*)=\emptyset$.
Otherwise, let $u_0\in A_1$ be the common vertex.
If $\{u_0,w_1^*,w_2^*\}$ induces a $3$-cycle, then $u_0w_1^*w_2^*$ a leaf block of $G^*$. Thus  $G'=G^*-u_0w_1^*-u_0w_2^*+u^*w_1^*+u^*w_2^*$ is minimally 2-edge-connected and $\rho(G')> \rho(G^*)$ by Lemma \ref{le-2.9}. Therefore,  there exists another vertex $u_2\in A_1$ that is adjacent to at least one of $\{w_1^*, w_2^*\}$. It follows that
$$\left\{\begin{array}{ll}
C=w_2^*u_0u^*u_2w_1^* \mbox{ is a cycle with the chord $u_0w_1^*$ }& \mbox{ if $u_2\sim w_1^*$, }\\
C=u^*u_0w_1^*w_2^*u_2 \mbox{ is a cycle with the chord $u_0w_2^*$ }& \mbox{ if $u_2\sim w_2^*$, }\\
\end{array}\right.
$$
which always leads a contradiction.

Secondly, we will show $d_{A_1}(w_1^*)=1$ and $d_{A_1}(w_2^*)\ge1$. In fact,
since $\delta(G^*)=2$ and $N_{A_2}({w_1^*})=N_{A_2}({w_2^*})=\emptyset$ by Claim \ref{claim-4.2},
we have $d_{A_1}(w_1^*), d_{A_1}(w_2^*)\ge 1$.
If $d_{A_1}(w_1^*), d_{A_1}(w_2^*)\ge 2$, then $G^*$ contains  $H(2, 2)$ (see Fig.\ref{fig-2}) as a subgraph.
We see that $H(2, 2)$ is not minimally 2-edge-connected, which contradicts   Lemma \ref{le-2.11}.

Combining the above two facts, we have $G^*[\{u^*, w_1^*, w_2^*\}\cup N_{A_1}(\{w_1^*, w_2^*\})]\cong H(1,t_2)$  for some positive $t_2\ge 1$,  and clearly
$H(1,t_2)$ is a leaf block of $G^*$.
Finally,  it suffices to show $t_2= 1$.
Suppose to the contrary  that  $t_2\ge 2$. Now let $N_{A_1}(w_1^*)=\{u_0\}$ and $N_{A_1}(w_2^*)=\{u_1, u_2, \ldots,u_{t_2}\}$.
Then $G'=G^*-w_1^*w_2^*+w_1^*u^*$
is a graph obtained from $G$ by replacing the  block $H(1,t_2)$ to $K_{2,t_2}*K_3$.
Thus, $G'$ is also minimally 2-edge-connected. By Lemma \ref{le-2.9},  we have $\rho(G')>\rho(G^*)$, a contradiction.
\end{proof}

Denote by $B_1$ the set of all isolated vertices in $B$. By Claim \ref{claim-4.2'}, $B=B_1\cup\{w_1^*w_2^*\}$.
By Claim \ref{claim-4.3}, we may assume that  $N_{A_1}(w_1^*)=\{v_1^*\}$ and $N_{A_1}(w_2^*)=\{v_2^*\}$ in what follows.

\begin{clm}\label{claim-4.3'}
$N_{B_1}(\{v_1^*, v_2^*\})=\emptyset$.
\end{clm}
\begin{proof}
Suppose to the contrary that there exists a vertex $w$ in $B_1$ with neighbor $v_1^*$ or $v_2^*$. Without loos of generality, we  assume  $N_{A_1}(w)=v_1^*$.
Notice that $d(w)\ge 2$.  We claim that $N_{A_1}(w)=\{v_1^*, v_2^*\}$. Since otherwise,
 $w\sim v\in A_1\backslash\{v_1^*, v_2^*\}$, then $u^*vwv_1^*w_1^*w_2^*v_2^*$ is  a cycle with the chord $v_1^*u^*$, a contradiction.
Thus,  $W=\{u^*,w,w_1^*,w_2^*,v_1^*,v_2^*\}$ induces a minimally  2-edge-connected leaf block of $G^*$.
By symmetry,  we also have $x_{v_1^*}=x_{v_2^*}$.
Let  $G'=G^*-w_2^*v_2^*+w_2^*v_1^*$.
We see that
$W$  also induces a   leaf block of  $G'$ and so
 $G'$ is also a minimally 2-edge connected graph.
 By Lemma \ref{le-2.9},  we have $\rho(G')>\rho(G^*)$,  a contradiction.
\end{proof}

Let $A_1'=A_1\backslash\{v_1^*,v_2^*\}$. By Claim \ref{claim-4.3'}, each vertex of \textcolor[rgb]{0.44,0.00,0.94}{$B_1$} only joins some vertices in $A_1'$. i.e. $N_{A_1}(w_i)=N_{A_1'}(w_i)$ for any $w_i\in B_1$.

\begin{clm}\label{claim-4.5}$|B_1|\ge 2$ and
$|N_{A_1}(w_1)\cap N_{A_1}(w_2)|=0$ or $2$ for any $w_1\not=w_2\in B_1$.
\end{clm}
\begin{proof}Firstly, we show that $|B_1|\ge 2$. Otherwise,
we may assume $B_1=\{w_1\}$, then
$G^*[N_{A_1}(w_1)\cup \{w_1,u^*\}]\cong K_{2, a_1}$, where $a_1=|N_{A_1}(w_1)|$.
By Claims \ref{claim-4.2}, \ref{claim-4.3} and \ref{claim-4.3'}, we have $G^*\cong F_1(t_1,a_2,0)$ for some positive $t_1\ge0$, $a_2\ge2$ and $3t_1+2a_2+7=m(G^*)=m$. By Proposition \ref{claim-4.9}, we know that
$\rho(G^*)=\rho(F_1(t_1,a_2,0))<\rho(F_0(t_1+1,a_2+1,0)$, a contradiction.

Suppose that $|N_{A_1}(w_1)\cap N_{A_1}(w_2)|\ge 3$, let $\{v_1,v_2,v_3\}\subseteq N_{A_1}(w_1)\cap N_{A_1}(w_2)$, then  $G^*$ contains a $5$-cycle $C_2=v_1w_1v_2w_2v_3$ with  chord $v_1w_2$, a contradiction.
Next we show that $|N_{A_1}(w_1)\cap N_{A_1}(w_2)|\not=1$. Otherwise, let $N_{A_1}(w_1)\cap N_{A_1}(w_2)=\{v\}$, then  $G^*$ contains a $6$-cycle $C_1=u^*w_1'w_1vw_2w_2'$ with  chord $u^*v$, where $w_1'\in N_{A_1}(w_1)\backslash v$ and $w_2'\in N_{A_1}(w_2)\backslash v$, it is a contradiction.
Thus,  we have $|N_{A_1}(w_1)\cap N_{A_1}(w_2)|=0$ or $2$.
\end{proof}

Notice that $d_{A_1}(w_i)\ge2$ for any $w_i\in B_1$. By Claims \ref{claim-4.1}-\ref{claim-4.5} we can get the structure of $G^*$ shown as Fig.\ref{G^*}.  In particular, if $e(B)=\emptyset$, then $G^*$ contains no  5-cycle $C=u^*v_1^*w_1^*w_2^*v_2^*$.
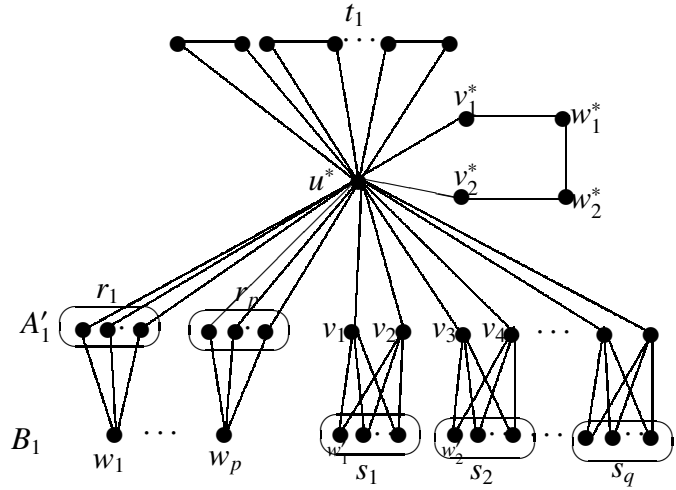
\begin{figure}[h]
\begin{center}
\centering \setlength{\unitlength}{2.4pt}
\unitlength 2.1mm 
\linethickness{0.4pt}
\ifx\plotpoint\undefined\newsavebox{\plotpoint}\fi 
\begin{picture}(40.888,29.116)(0,0)
\put(5.256,9.355){\oval(6.307,2.523)[]}
\put(13.454,9.144){\oval(6.307,2.523)[]}
\put(21.758,2.522){\oval(6.307,2.523)[]}
\put(29.011,2.312){\oval(6.307,2.523)[]}
\put(37.735,2.102){\oval(6.307,2.523)[]}
\multiput(3.364,9.355)(.03350725,-.1035942){69}{\line(0,-1){.1035942}}
\multiput(5.676,2.207)(-.0323077,.5498462){13}{\line(0,1){.5498462}}
\put(5.256,9.355){\line(0,1){0}}
\multiput(5.676,2.207)(.03355319,.14985106){47}{\line(0,1){.14985106}}
\put(7.253,9.25){\line(0,1){0}}
\multiput(11.562,9.25)(.0328438,-.2102188){32}{\line(0,-1){.2102188}}
\multiput(12.613,2.523)(.032875,.4204375){16}{\line(0,1){.4204375}}
\put(13.139,9.25){\line(1,0){.105}}
\multiput(12.613,2.523)(.03338824,.08037647){85}{\line(0,1){.08037647}}
\multiput(19.761,2.523)(.0326207,.2319655){29}{\line(0,1){.2319655}}
\multiput(20.707,9.25)(.0334545,-.3201364){22}{\line(0,-1){.3201364}}
\multiput(21.443,2.207)(.045,.030143){7}{\line(1,0){.045}}
\multiput(20.812,9.355)(.03332927,-.08717073){82}{\line(0,-1){.08717073}}
\multiput(23.545,2.207)(.0315,.7148){10}{\line(0,1){.7148}}
\multiput(23.86,9.355)(-.03356944,-.09927778){72}{\line(0,-1){.09927778}}
\put(21.443,2.207){\line(1,0){.105}}
\multiput(19.761,2.418)(.033598361,.055139344){122}{\line(0,1){.055139344}}
\multiput(27.644,9.25)(-.0334091,-.3249091){22}{\line(0,-1){.3249091}}
\multiput(28.485,2.102)(-.03364,.27752){25}{\line(0,1){.27752}}
\multiput(27.644,9.04)(.033597938,-.071525773){97}{\line(0,-1){.071525773}}
\put(30.903,2.102){\line(0,1){6.727}}
\multiput(26.909,2.207)(.033563025,.060067227){119}{\line(0,1){.060067227}}
\multiput(28.485,2.207)(.03358333,.09636111){72}{\line(0,1){.09636111}}
\multiput(35.423,1.997)(.03302857,.21022857){35}{\line(0,1){.21022857}}
\multiput(36.579,9.355)(.0323077,-.5740769){13}{\line(0,-1){.5740769}}
\multiput(36.999,1.892)(.03369231,.09298718){78}{\line(0,1){.09298718}}
\put(39.627,9.145){\line(0,-1){7.253}}
\multiput(39.627,1.892)(-.033494505,.079703297){91}{\line(0,1){.079703297}}
\put(36.579,9.145){\line(0,1){0}}
\multiput(35.633,2.102)(.033598361,.057729508){122}{\line(0,1){.057729508}}
\multiput(3.364,9.04)(.0613451957,.0336654804){281}{\line(1,0){.0613451957}}
\multiput(20.602,18.5)(-.0564191176,-.0336213235){272}{\line(-1,0){.0564191176}}
\multiput(7.253,9.25)(.0486263345,.0336654804){281}{\line(1,0){.0486263345}}
\put(20.917,18.71){\line(-1,-1){9.355}}
\put(11.562,9.355){\line(0,1){0}}
\multiput(13.139,9.145)(.033708861,.04035865){237}{\line(0,1){.04035865}}
\multiput(21.128,18.71)(-.033616279,-.055){172}{\line(0,-1){.055}}
\multiput(20.812,9.25)(.0323846,.7276923){13}{\line(0,1){.7276923}}
\multiput(21.233,18.71)(.03332927,-.11920732){82}{\line(0,-1){.11920732}}
\multiput(21.022,18.71)(.033614213,-.050690355){197}{\line(0,-1){.050690355}}
\put(27.644,8.724){\line(0,1){0}}
\multiput(21.022,18.71)(.0340912162,-.0337364865){296}{\line(1,0){.0340912162}}
\multiput(21.233,18.605)(.0538362369,-.0336933798){287}{\line(1,0){.0538362369}}
\multiput(21.128,18.71)(.0634948805,-.0337235495){293}{\line(1,0){.0634948805}}
\multiput(21.022,18.815)(.057086207,.033525862){116}{\line(1,0){.057086207}}
\put(27.644,22.704){\line(1,0){6.727}}
\multiput(21.128,18.605)(.033585799,.052242604){169}{\line(0,1){.052242604}}
\put(26.804,27.434){\line(-1,0){3.994}}
\multiput(22.809,27.434)(-.03371698,-.16658491){53}{\line(0,-1){.16658491}}
\multiput(21.022,18.605)(-.03343182,.20065909){44}{\line(0,1){.20065909}}
\put(19.551,27.434){\line(-1,0){4.204}}
\multiput(15.346,27.434)(.03364,-.052857143){175}{\line(0,-1){.052857143}}
\multiput(21.233,18.184)(-.033657895,.040109649){228}{\line(0,1){.040109649}}
\put(13.665,27.434){\line(-1,0){4.204}}
\multiput(9.46,27.434)(.0431927273,-.0336363636){275}{\line(1,0){.0431927273}}
\put(20.917,18.815){\line(5,-1){6.832}}
\put(27.75,17.449){\line(1,0){6.412}}
\put(34.161,17.449){\line(0,1){5.045}}
\put(3.574,9.15){\circle*{1}}
\put(5.15,9.15){\circle*{1}}
\put(7.253,9.15){\circle*{1}}
\put(11.562,9.04){\circle*{1}}
\put(13.244,9.04){\circle*{1}}
\put(15.136,9.04){\circle*{1}}
\put(20.602,9.04){\circle*{1}}
\put(23.86,9.04){\circle*{1}}
\put(27.644,8.829){\circle*{1}}
\put(30.693,8.829){\circle*{1}}
\put(36.579,8.829){\circle*{1}}
\put(39.522,8.829){\circle*{1}}
\put(39.522,2.323){\circle*{1}}
\put(37.105,2.323){\circle*{1}}
\put(35.423,2.323){\circle*{1}}
\put(30.798,2.523){\circle*{1}}
\put(28.59,2.523){\circle*{1}}
\put(27.119,2.523){\circle*{1}}
\put(23.545,2.523){\circle*{1}}
\put(21.338,2.523){\circle*{1}}
\put(19.866,2.523){\circle*{1}}
\put(12.508,2.523){\circle*{1}}
\put(5.571,2.523){\circle*{1}}
\put(21.022,18.5){\circle*{1}}
\put(27.855,22.494){\circle*{1}}
\put(27.539,17.554){\circle*{1}}
\put(33.951,22.494){\circle*{1}}
\put(34.161,17.554){\circle*{1}}
\put(9.565,27.3){\circle*{1}}
\put(13.665,27.3){\circle*{1}}
\put(15.241,27.3){\circle*{1}}
\put(19.551,27.3){\circle*{1}}
\put(22.914,27.3){\circle*{1}}
\put(26.804,27.3){\circle*{1}}
\put(20.812,29.116){\makebox(0,0)[cc]{$t_1$}}
\put(27.85,23.755){\makebox(0,0)[cc]{$v_1^*$}}
\put(27.85,18.7){\makebox(0,0)[cc]{$v_2^*$}}
\put(35.45,22.599){\makebox(0,0)[cc]{$w_1^*$}}
\put(35.45,17.238){\makebox(0,0)[cc]{$w_2^*$}}
\put(5.15,11.457){\makebox(0,0)[cc]{$r_1$}}
\put(13.98,11.142){\makebox(0,0)[cc]{$r_p$}}
\put(5.256,.726){\makebox(0,0)[cc]{$w_1$}}
\put(12.613,.726){\makebox(0,0)[cc]{$w_p$}}
\put(8.619,2.418){\makebox(0,0)[cc]{$\cdots$}}
\put(21.548,0){\makebox(0,0)[cc]{$s_1$}}
\put(28.801,0){\makebox(0,0)[cc]{$s_2$}}
\put(37.735,0){\makebox(0,0)[cc]{$s_q$}}
\put(33.531,8.829){\makebox(0,0)[cc]{$\cdots$}}
\put(33.11,2.312){\makebox(0,0)[cc]{$\cdots$}}
\put(21.222,27.329){\makebox(0,0)[cc]{$\cdots$}}
\put(.631,9.355){\makebox(0,0)[cc]{$A_1'$}}
\put(0,2.102){\makebox(0,0)[cc]{$B_1$}}
\put(18.71,18.71){\makebox(0,0)[cc]{$u^*$}}
\put(19.536,9.){\makebox(0,0)[cc]{$v_1$}}
\put(22.699,9.){\makebox(0,0)[cc]{$v_2$}}
\put(26.473,9){\makebox(0,0)[cc]{$v_3$}}
\put(29.627,9.){\makebox(0,0)[cc]{$v_4$}}
\put(19.746,1.416){\makebox(0,0)[cc]{\scriptsize$w_1'$}}
\put(26.999,1.416){\makebox(0,0)[cc]{\scriptsize$w_2'$}}
\put(22.374,2.35){\makebox(0,0)[cc]{$\cdots$}}
\put(29.627,2.35){\makebox(0,0)[cc]{$\cdots$}}
\put(38.246,2.35){\makebox(0,0)[cc]{$\cdots$}}
\put(6.187,9.079){\makebox(0,0)[cc]{$\cdots$}}
\put(14.175,9.079){\makebox(0,0)[cc]{$\cdots$}}
\end{picture}
\vspace{-0.4cm}
\end{center}
\caption{\footnotesize{The structure  of $G^*$, where $t_1\ge0$, $p+ \sum_{i=1}^qs_i=|B_1|$ and $\sum_{i=1}^pr_i+ 2q=|A_1'|$.}}\label{G^*}
\end{figure}
For nonnegative  integer $p$, $q$,  $G^*[N_{A_1}(B_1)\cup B_1]\cong \bigcup_{i=1}^pK_{1,r_i}\bigcup_{j=1}^qK_{2,s_j}$ and satisfy  $3t_1+2\sum_{i=1}^pr_i+2\sum_{i=1}^qs_i+2q+5=m$ ($r_i$, $s_i\ge2$).
Furthermore, we will determine the values of $p,q$.
\begin{clm}\label{claim-4.4}
$p\le1$ and $q\le1$.
\end{clm}
\begin{proof}
We firstly show $p\le1$. Suppose $p\ge2$, then there exists  two vertices, say $w_1, w_2$ in $B_1$ with $G^*[N_{A_1'}(w_i)\cup\{w_i\}=K_{1,r_i}$ for $i=1,2$.
Without loss of generality, we may assume that $x_{w_1} \ge x_{w_2}$.
Denote by
$$G^{'}=G^*-\sum_{v \in N_{A_1}(w_2)}vw_2+\sum_{v \in N_{A_1}(w_2)}vw_1.$$
By Lemma \ref{le-2.9},  we have $\rho(G^{'})> \rho^*$.
Clearly  $w_2$ is an isolated vertex of $G^{'}$.
Set $G^{''}=G^{'}-\{w_2\}$.
Then $G^{''}$ is also a minimally 2-edge-connected graph
 since $N_{A_1}(w_1)\cup N_{A_1}(w_2)\cup\{w_1, u^*\}$ induces a block $K_{2,r_1+r_2}$ in $G^{''}$.
However $\rho(G^{''})=\rho(G^{'})> \rho^*$, a
contradiction.

Now we will show $q\le1$.
Otherwise, $q\ge2$. Then $G^*[N_{A_1}(B_1)\cup B_1]$ contains $K_{2,s_1}, K_{2,s_2}$ ($s_1, s_2\ge 2$) as induced subgraphs.
Denote by $w_i'\in V(K_{2,s_i})\cap B_1$ for $i=1,2$.
Set $N_{A_1}(w_1')=\{v_1,v_2\}$ and $N_{A_1}(w_2')=\{v_3, v_4\}$.
Let $X$ be the Perron vector of $G^*$ whose entry $x_v$ is labelled by vertex $v$.
By the symmetry, $x_{v_1}=x_{v_2} $ and $x_{v_3}=x_{v_4}$.
Without loss of generality, we may assume that  $x_{v_1} \ge x_{v_3}$.
Then $x_{v_1}+x_{v_2}\ge x_{v_3}+x_{v_4}$.
Let $G^{'''} =G^*-w_2'v_3-w_2'v_4+w_2'v_1+w_2'v_2$, and  $\rho'''=\rho(G^{'''})$.
Clearly, $G^{'''}$ is   minimally 2-edge-connected.
By Lemma \ref{7}, we get $\rho'''>\rho^*$, which contradicts with the maximality of $\rho^*$.
\end{proof}

By comparing Fig. \ref{fig-5} and Fig. \ref{G^*}, we have $t_2=r_1$ and  $t_3=s_1$.
From Claims \ref{claim-4.1}-\ref{claim-4.4}, we know that
$G^*$ has two forms: $F_0(t_1,t_2,t_3)$  or $F_1(t_1',t_2',t_3')$, where
$t_1,t_3, t_1',t_3'\ge 0$, $t_2,t_2'=0$ or $t_2,t_2'\ge 2$ and satisfy
\begin{equation*}m=
\left\{\begin{array}{ll}
3t_1+2t_2& \mbox{ if  $t_3=0$}\\
3t_1+2t_2+2t_3+2& \mbox{ if   $t_3\ge 1$}\\
3t_1'+2t_2'+5& \mbox{ if  $t_3'=0$}\\
3t_1'+2t_2'+2t_3'+2+5& \mbox{ if $t_3'\ge 1$}.
\end{array}\right.
\end{equation*}
Clearly, $t_1\geq1$ since otherwise $m$ is even.
Suppose that $G^*\cong F_1(t_1',t_2',t_3')$,
by Proposition \ref{claim-4.9}, we have $\rho(F_1(t_1',t_2',t_3'))<\rho(F_0(t_1'+1,t_2'+1,t_3'))$,
which contradicts the maximality of $\rho(G^*)$.
Thus
$G^*\cong F_0(t_1,t_2,t_3)$,
where
$t_3\ge 0$, $t_2=0$ or $\ge 2$ and satisfy
\begin{equation*}m=
\left\{\begin{array}{ll}
3t_1+2t_2& \mbox{ if  $t_3=0$}\\
3t_1+2t_2+2t_3+2& \mbox{ if $t_3\ge 1$}.
\end{array}\right.
\end{equation*}
If $t_3=0$, then $t_2=\frac{m-3t_1}{2}$, and thus  $G^*\cong F_0(t_1,t_2,0)=F_0(t_1,\frac{m-3t_1}{2},0))$.
If  $t_3\ge 1$, then $t_2+t_3+1=\frac{m-3t_1}{2}$.
By  Proposition \ref{claim-4.9'},  we have
\begin{equation*}
\rho\left(F_0(t_1,t_2,t_3)\right)
<\rho\left(F_0(t_1,t_2+t_3+1,0)\right)
=\rho\left(F_0(t_1,\frac{m-3t_1}{2},0)\right).
\end{equation*}
By the maximality of $\rho(G^*)$ again, we get $G^*\cong F_0(t_1,\frac{m-3t_1}{2},0)$ for some $t_1\ge 1$ and $\frac{m-3t_1}{2}=0$ or $\ge2$.
At last, we will show $t_1=1$.
\begin{clm}\label{claim-4.8}
$ G^*\cong F_0(1,\frac{m-3 }{2},0)$ for $m\ge11$ and $m\neq 15$, and $G^*\cong F_5$ for $m=15$.
\end{clm}
\begin{proof}
Denote  by $f(G, x)$  the characteristic polynomial of quotient matrix of  $A(G)$.
Let
$$f_1(x)=f(F_0(t_1,\frac{m-3t_1 }{2},0),x)=x^4-x^3+(t_1-m)x^2+(m-3t_1)x-3{t_1}^2+m t_1,$$
where $t_1\ge 2$, and let
$f_2(x)=f(F_0(1,\frac{m-3}{2},0), x)=x^4-x^3+(1-m)x^2+(m-3)x+m-3.$
Then
\begin{equation}\label{last}
f_1(x)-f_2(x)=(t_1-1)x^2+(3-3t_1)x-3{t_1}^2+3+(t_1-1)m=(t_1-1)g(x),\end{equation}
where $g(x)=x^2-x+m-3(t_1+1)$.

If $\frac{m-3t_1}{2}\ge 2$, then $2\le t_1\le \frac{m-4}{3}$. Note that $m\ge 11$.
One can verify that   $g(x)>0$ for  $x >\sqrt{m-2}$.
From (\ref{last}), we have  $f_1(x)-f_2(x)>0$ for $x >\sqrt{m-2}$.
Notice that $f_2(\sqrt{m-2})=-\sqrt{m-2}-1<0$.
By Lemma \ref{lem-2.6'} (ii),
we have $\rho(F_0(t_1, \frac{m-3t_1 }{2},0))< \rho(F_0(1,\frac{m-3}{2},0))$ for any $t_1\ge 2$. Thus, $G^*\cong F_0(1,\frac{m-3}{2},0)$.

If  $\frac{m-3t_1 }{2}=0$,  then $t_1=\frac{m}{3}$, i.e. $G^*\cong F_0(\frac m3, 0, 0),$
where $3\mid m$ and $m\ge 11$.
 By the computation,
 $$\rho\left(F_0\left(\frac m3, 0, 0\right)\right)= \frac{1+\sqrt{1+\frac{8}{3}m}}{2} < \sqrt{m-2}<\rho\left(F_0\left(1,\frac{m-3}{2},0\right)\right)$$ for $m \ge 19$.
 For $11\le m\le 17$, i.e. $m=15$,  we have $\frac{1+\sqrt{41}}{2}=\rho(F_5)>\rho(F_0(1, 6,0))$.
 By the above arguments, $G^*\cong F_5$ for $m=15$, and
 $G^*\cong F_0(1,\frac{m-3}{2},0)$  for $m\ge 11$ and $m\neq 15$.
\end{proof}

Notice that $F_0(1,\frac{m-3}{2},0)\cong K_{2,\frac{m-3}{2}}*K_3$ and $\rho(F_0(1,\frac{m-3}{2},0))$ is the largest root of $x^4-x^3+(1-m)x^2+(m-3)x+m-3=0$.
It completes the proof of Theorem \ref{thm-1.2} (ii).
\end{proof}
\begin{remark}For odd $m<11$,
by Claims \ref{claim-4.1}-\ref{claim-4.4}, we get that the minimally 2-edge connected graph and the extremal graph
$G^*$ is given by the following Tables.
\begin{table}[h]
\centering
\renewcommand{\arraystretch}{1.1}
\setlength{\tabcolsep}{7mm}
\begin{tabular}{c|c|c|c}
\hline
$m$      &   Minimally 2-edge-connected graph       &   $G^*$   &$\rho(G^*)$   \\ \hline
$3$      &   $C_3$                                  &   $C_3$              &2      \\
$5$      &    $C_5$                                 &   $C_5$              &2      \\
$7$     &   $C_7$, $SK_{2,3}$, $C_3*C_4$               &   $C_3*C_4$        &2.5035    \\
$9$     &  $C_9$, $SK_{2,4}$, $C_3*C_6$, $C_3*K_{2,3}$, $C_4*C_5$, $F_3$  &   $F_3$  &3  \\
\hline
\end{tabular}
\end{table}
\end{remark}


\end{document}